\documentclass[11pt]{article}
\usepackage{amsmath,amsthm,amssymb,amsfonts,mathtools,MnSymbol}
\usepackage{epsfig,epic}
\usepackage{geometry}
\usepackage{graphicx}
\usepackage{xcolor}
\usepackage{psfrag}
\usepackage{leftidx,tensor}
\usepackage{enumerate}
\usepackage[colorinlistoftodos]{todonotes}

\usepackage{color}

\newtheorem{satz}{Satz}
\newtheorem{theorem}[satz]{Theorem}
\newtheorem{lemma}[satz]{Lemma}
\newtheorem{defi}[satz]{Definition}

\newtheorem{rem}[satz]{Remark}

\def \R{\mathbb{R}}

\def \N{\mathbb{N}}

\def \r{\mathcal{R}}

\def \Hext{H_{\operatorname{ext}}}
\def \Heff{H_{\operatorname{eff}}}

\def \tp{\;\vdots\;}

\def \H{{\bf H}}

\renewcommand\llangle{\mathchoice{\big\langle\hspace{-.3em}\big\langle}
           {\langle\hspace{-.2em}\langle}{\langle\!\langle}{\langle\!\langle}}
\renewcommand\rrangle{\mathchoice{\big\rangle\hspace{-.3em}\big\rangle}
           {\rangle\hspace{-.2em}\rangle}{\rangle\!\rangle}{\rangle\!\rangle}}

\allowdisplaybreaks 

\begin{document}

\title{Existence of weak solutions to an evolutionary model for magnetoelasticity}

\author{Barbora Bene\v{s}ov\'{a}\footnote{Chair of Mathematics in the Sciences, University of W{\"u}rzburg, Emil-Fischer-Str. 40, 97074 W{\"u}rzburg, Germany, \texttt{barbora.benesova@mathematik.uni-wuerzburg.de}, \texttt{johannes.forster@mathematik.uni-wuerzburg.de}, \texttt{anja.schloemerkemper@mathematik.uni-wuerzburg.de}},
Johannes Forster$^*$,
Chun Liu\footnote{Department of Mathematics, The Penn State University, University Park, PA 16802, USA, \texttt{liu@math.psu.edu}},
and Anja Schl\"omerkemper$^*$
}

\date{\today}

\maketitle

{\abstract{We prove existence of weak solutions to an evolutionary model derived for magnetoelastic materials. The model is phrased in Eulerian coordinates and consists in particular of (i) a Navier-Stokes equation that involves magnetic and elastic terms in the stress tensor obtained by a variational approach, of (ii) a regularized transport equation for the deformation gradient and of (iii) the Landau-Lifshitz-Gilbert equation for the dynamics of the magnetization.
The proof is built on a Galerkin method and a fixed-point argument. It is based on ideas from F.-H.~Lin and the third author for systems modeling the flow of liquid crystals as well as on methods by G.~Carbou and P.~Fabrie for solutions of the Landau-Lifshitz equation.
}}\\

\section{Introduction}

Magnetoelastic (or magnetostrictive) materials respond elastically to an applied magnetic field (magnetostriction) and/or react with a change of magnetization to a mechanical stress (magnetoelastic effect). Because of the remarkable response to external stimuli, they are smart materials that are attractive not only from the point of view of mathematical modeling but also for applications. Magnetoelastic materials are among others used in sensors to measure force or torque (cf., e.g., \cite{BienkowskiSzewczyk2002, BienkowskiSzewczyk2004, GrimesRoy_etal2011}) as well as magnetic actuators (cf., e.g., \cite{SnyderNguyenRamanujan2010}) or generators for ultrasonic sound (cf., e.g., \cite{BuchelnikovVasilev1992}).

Modeling of magnetoelastic materials goes back to \cite{brown} as well as  \cite{Tiersten1964, Tiersten1965}. Later, many works appeared studying magnetoelasticity particularly in the static case relying on a minimization of energy, see, e.g., \cite{DeSimoneDolzmann1998,DeSimoneJames2002,JamesKinderlehrer1993}. Let us point out that magnetoelastic models can be seen as generalizations of models for micromagnetics that are also studied for their own right, cf., e.g., the reviews \cite{KruzikProhl2006,DKMO}. In the dynamic case, the available works confine themselves to the small strain setting, cf., e.g., \cite{ChipotShafrir_etal2009,CarbouEfendievFabrie2011}. 

The prominent difficulty in analyzing magnetoelastic models lies in the fact that while elasticity is commonly formulated in the reference configuration, micromagnetics is modeled in the current or deformed configuration. To overcome this issue, models are either formulated in the small-strain setting as in, e.g., \cite{ChipotShafrir_etal2009,CarbouEfendievFabrie2011} or, because the elastic energy assures invertibility of the deformation, it is possible to transform the magnetic part into the reference configuration \cite{DeSimoneDolzmann1998,DeSimoneJames2002,KruzikStefanelliZeman2015}. 

In this article, we shall take a different approach and formulate the fully nonlinear problem of magnetoelasticity completely in Eulerian coordinates in the current configuration. In the current configuration, the main state variable is the velocity and not the deformation. This poses an obstacle from the point of view of elasticity since then the deformation gradient is not readily available. Thus, we follow the approach of \cite{LiuWalkington2001} where this issue has been resolved by finding a differential equation---a transport equation for the deformation gradient---that allows to obtain the deformation gradient (in the current configuration) from the velocity gradient. Therefore, we will not need to care about the invertibility of the deformation. Moreover, the model is perfectly fitted to be used in modeling of so-called magnetorheological fluids; cf. e.g. \cite{Wereley2014}. Those are so-called smart fluids containing magnetoelastic particles in a carrier fluid. Indeed, it seems feasible that the system of partial differential equations under consideration \eqref{sumeqmotshortsimple}--\eqref{summicroforcebalancesimple} can be extended to fluid models via a phase field approach (cf. also \cite{LiuWalkington2001}).

As for the magnetic part, we model the evolution of magnetization by the Landau-Lifshitz-Gilbert (LLG) equation \cite{LandauLifshitz1935,Gilbert1955,Gilbert2004} with, however, the time derivative replaced by the convective one. This is in order to take into account that changes of the magnetization also occur if transported by the underlying viscoelastic material. We refer to Section \ref{sec:Model} for a detailed description of the model, see also \cite{ForsterGarcia-Cervera_etal2016} and \cite{Forster2016}. In this work, we prove existence of weak solutions in the case where the stray field and the anisotropy are neglected for mathematical reasons, and where we regularized the evolution equation for the deformation gradient, cf.\ also \cite{Forster2016} for the case that the external magnetic field is zero in addition. Our proof is based on a Galerkin method discretizing the velocity in the balance of momentum equation and a fixed point argument. It borrows ideas from \cite{LinLiu1995}, beyond which our system is further coupled to the evolution of the deformation gradient and the LLG equation. For the treatment of the LLG equation, however, we further utilize methods from \cite{CarbouFabrie2001}, which are necessary to converge approximate solutions using higher regularity estimates, see also \cite{BenesovaForster_etal2016c} for a sketch of the proof and an announcement of this work.

As an aside, we remark that a corresponding system which is equipped with a gradient flow for the magnetization $M$ instead of the LLG equation is studied in \cite{Forster2016}. This system has the advantage of being closer to the system studied in \cite{LinLiu1995} in terms of the magnetization. The gradient flow type dynamics are less involved than the LLG equation, which makes the treatment of the equation for the magnetization $M$ a lot easier. However, in the context of micromagnetics, the LLG equation is the established description of the dynamics of the magnetization. For the gradient flow case, existence of weak solutions is proved by a Galerkin approximation and a fixed-point argument similar to the proofs of this paper, but less regularity for the magnetization is needed.

The paper is structured as follows: we start with a presentation of the considered model for magnetoelastic materials in Section~\ref{sec:Model}. There, we state the model equations and give a brief derivation. In Section~\ref{sec:MainResult}, we state the main result of this article, viz the existence of weak solutions to the evolutionary model for magnetoelasticity in Theorem~\ref{ThmSln}. The proof of this Theorem is presented in Section~\ref{sec-proof}. In Section~\ref{sec:PfsLem}, we prove two lemmas used in the proof of Theorem~\ref{ThmSln}.

\section{Presentation of the model}
\label{sec:Model}
Let $\Omega \subset \R^d$, $d=2,3$  represent the current configuration. Then we consider the following model for magnetoelastic solids:
\begin{align}
& \partial_t v + (v\cdot \nabla) v - \mathrm{div}\mathcal{T} = f \label{sumeqmotshortsimple}& & \text{(balance of momentum)} \\[0.5em]
& \nabla\cdot v = 0, \label{incompressibility}  & & \text{(incompressibility)} \\[0.5em]
& \partial_t F + (v\cdot\nabla)F - \nabla v F = \kappa \Delta F, \label{chainruleregular}& &\text{(evolution of deform. gradient)} \\[0.5em]
& \partial_t M + (v\cdot\nabla) M = - \gamma M\times H_\mathrm{eff} - \lambda M \times M \times H_\mathrm{eff} ,\label{summicroforcebalancesimple} & &\text{(LLG equation)}
\end{align}
closed by boundary conditions \eqref{simpleboundaryv}--\eqref{boundaryMsimple} and initial conditions \eqref{initialv}--\eqref{initialM} below. Here, \eqref{sumeqmotshortsimple} is the balance of momentum in Eulerian coordinates with  $v: \Omega \times (0,T) \to \R^d$ being the velocity mapping, $\mathcal{T}$ the stress tensor and $f$ the applied body forces. Similarly, \eqref{summicroforcebalancesimple} is a variant of the  Landau-Lifschitz-Gilbert (LLG) evolution equation for the magnetization $M: \Omega \times (0,T) \to \R^3$, in which we replaced the time-derivatives in the LLG equation by the convective one in order to take changes of the magnetization through transport into account. In this equation, $H_\mathrm{eff}$ is the effective magnetic field, cf.\ \eqref{Heff} below,  $\gamma>0$ is the electron gyromagnetic ratio and $\lambda>0$ is a phenomenological damping parameter. Here and in the following, we impose the standard constraint
\begin{equation}
|M| = 1 \mbox{ almost everywhere in } \Omega \times (0,T).
\end{equation}
Equation \eqref{chainruleregular} is an evolution for $F$ which, in our modeling, is an approximation for the deformation gradient in Eulerian coordinates. Indeed, if $\kappa = 0$, \eqref{chainruleregular} is obtained by taking a time derivative of the deformation gradient and rephrasing it in Eulerian coordinates, cf. \cite[Equation~(5)]{LiuWalkington2001}. In this case, \eqref{chainruleregular} is an evolution equation for the deformation gradient, but taking $\kappa = 0$ would make the proof of existence more involved and cannot be done without further assumptions on $F$. Therefore, we include a regularization term (cf., e.g., \cite[p.~1461]{LinLiuZhang2005}) with $\kappa$ presumably small. \  

The stress-tensor $\mathcal{T}$ as well as the effective field $H_\mathrm{eff}$ are constitutive quantities. In this work, we assume the decomposition
$$
\mathcal{T} = -p\mathbb{I} + \nu \nabla v + \mathcal{T}_\mathrm{rev},
$$
where $-p \mathbb{I}$ represents the pressure (which, however, shall not appear in our work since we will consider weak solutions only) and $\nu \nabla v$ is the viscous stress corresponding to a quadratic dissipation potential. Finally, $\mathcal{T}_\mathrm{rev}$ is the magnetoelastic part of the stress tensor that, as well as the effective magnetic field $\Heff$, will be deduced from the Helmholtz free energy. 

For the Helmholtz free energy in magnetoelasticity we have the following general form:
\begin{equation}
\psi(F,M) = \underbrace{A \int_{\Omega} |\nabla M|^2~dx}_{\text{exchange energy}} + \underbrace{\int_{\Omega} \phi(F,M)~dx}_{\text{anisotropy energy}} + 
\underbrace{\frac{\mu_0}2 \int_{\R^3} |H|^2~dx}
_{\text{stray field energy}} +  \underbrace{\int_{\Omega} W(F)~dx}_{\text{elastic energy}} \underbrace{ - \mu_0\int_\Omega M\cdot H_{\text{ext}} dx}_{\text{Zeeman energy}},
\label{energy}
\end{equation}
where the stray field $H:\R^3\to\R^3$ is obtained from (possibly a reduced set) of the magnetostatic Maxwell equations. Notice that the whole energy including its elastic part is formulated in the current configuration. From the Helmholtz free energy we obtain the effective field $H_\mathrm{eff}$ by taking the negative variational derivative of $\psi$ with respect to $M$. In order to obtain $\mathcal{T}_\mathrm{rev}$ we use that the elastic stress is a variational derivative of the Helmholtz free energy with respect to the deformation gradient $F$. However, care is needed during this procedure since the free energy has to be transferred back to the reference configuration and then the derivative with respect to the deformation gradient is taken in order to obtain the Piola-Kirchhoff stress tensor. This stress tensor is subsequently again transformed into the current configuration to obtain the Cauchy stress tensor. We present the derivation only for a simplified case considered in this article and refer to \cite{Forster2016,ForsterGarcia-Cervera_etal2016} for a detailed derivation of the simplified as well as the general model, which is based on taking variations of the action functional while carefully taking into account changes between the Eulerian and Lagrangian coordinates. 

Here, we study a simplified situation of isotropic magnetic particles (which allows us to set the anisotropy energy to zero). Further, we neglect the stray field energy (for mathematical reasons). Thus, we are left with 
\begin{equation}
\psi(F,M) = A \int_{\Omega} |\nabla M|^2~dx+  \int_{\Omega} W(F)~dx -\mu_0\int_\Omega M\cdot H_{\text{ext}} dx,
\label{energy-simpl}
\end{equation}
and so the effective magnetic field, which equals the negative variational derivative of $\psi$ with respect to $M$, is given by
\begin{equation} \label{Heff}
H_\mathrm{eff}= 2A \Delta M + \mu_0 H_{\text{ext}}.
\end{equation}
To obtain $\mathcal{T}_\mathrm{rev}$, we need to transform $\psi$ from \eqref{energy-simpl} to the reference configuration $\widetilde{\Omega}$. To this end, we define the deformation by the flow map $x:\widetilde{\Omega}\times[0,T]\to\Omega$, $(X,t)\mapsto x(X,t)$ and assume that $X\mapsto x(X,t)$ is a bijective mapping at every time $t\in[0,T]$. With the flow map, we define the velocity in the Eulerian coordinate system $v: \Omega\times[0,T] \to\R^d$ by
\begin{equation*}
v(x(X,t),t) = \frac{\partial}{\partial t}x(X,t).
\end{equation*}
We denote by $X\in\widetilde{\Omega}$ material points in the reference configuration (Lagrangian coordinates) and by $x\in\Omega$ spatial points in the current configuration (Eulerian coordinates). Moreover, we define $\widetilde{M}: \widetilde{\Omega}\times[0,T]\to\R^3$ to be the magnetization in the reference configuration satisfying $M(x(X,t),t) = \widetilde{M}(X,t)$, and $\widetilde{F}: \widetilde{\Omega}\times[0,T]\to\R^{d\times d}$ to be the deformation gradient in the reference configuration satisfying $F(x(X,t),t) = \widetilde{F}(X,t)$. Next, we obtain for the Helmholtz free energy transformed in Lagrangian coordinates, denoted by $\widetilde{\psi}(x,\widetilde{F}, \widetilde{M})$,
$$
\widetilde \psi(x,\widetilde{F}, \widetilde{M}) = \int_{\widetilde{\Omega}} A |\nabla_X \widetilde M (X) \widetilde F^{-1}(X,t)|^2  - \mu_0 \widetilde M(X,t){\cdot}H_{\text{ext}}(x(X,t),t) + W(\widetilde F(X,t)) ~\mathrm{d}X.
$$
Notice that, due to incompressibility \eqref{incompressibility}, the Jacobian of the transformation is one. Moreover, notice that through the external magnetic field, the Helmholtz free energy in the reference configuration also depends on the deformation itself. Thus, the term $\widetilde{\mathcal{F}}(x)= - \int_{\widetilde{\Omega}} \mu_0 \widetilde{M}(X)\cdot H_{\text{ext}}(x(X,t),t)~\mathrm{d}X$ can be understood as the potential of an applied volume force to the mechanical system (cf.\ forces with generalized potentials in, e.g., \cite{Ciarlet1988}), whence the volume force $\widetilde f$ is obtained as the negative variational derivative of $\widetilde{\mathcal{F}}$ with respect to $x$. Transforming back to the current configuration, we have that
$$
 f=\mu_0 \nabla H_\mathrm{ext}^\top M.
$$

Moreover, taking the variational derivative of $\widetilde \psi$ with respect to $\widetilde F$ and transforming back to the current configuration, we obtain for the elastic stress tensor 
$$
\mathcal{T}_\text{rev} = -2A\nabla M\odot\nabla M  + W'(F)F^\top \qquad \text{with} \quad (\nabla M\odot\nabla M)_{ij} = \sum_k \nabla_i M_k \nabla_j M_k. 
$$
Altogether, we are left with the following system of partial differential equations
\begin{align}
& \partial_t v + (v\cdot \nabla) v+ \nabla p + \nabla\cdot \bigl(2 A \nabla M\odot\nabla M - W'(F)F^\top)  - \nu \Delta v \label{sys1} = \mu_0 \nabla H_\mathrm{ext}^\top M  \\[0.5em]
& \nabla\cdot v = 0 \\[0.5em]
& \partial_t F + (v\cdot\nabla)F - \nabla v F = \kappa \Delta F, \label{sys2}\\[0.5em]
& \partial_t M + (v\cdot\nabla) M = - \gamma M\times \bigl(2A \Delta M + \mu_0 H_{\text{ext}}\bigr) - \lambda M\times M \times \bigl(2A \Delta M + \mu_0 H_{\text{ext}}\bigr)\label{sys3}
\end{align}
in $\Omega\times(0,T)$, accompanied with the following boundary/initial conditions:
\begin{eqnarray}
\label{simpleboundaryv}
v &=& 0 \qquad\text{on }\partial\Omega\times(0,T), \\
\label{simpleboundaryF}
F &=& 0 \qquad\text{on }\partial\Omega\times(0,T), \\
\label{boundaryMsimple}
\frac{\partial M}{\partial n} = (\nabla M) n &=& 0 \qquad\text{on }\partial\Omega\times(0,T)\\
v(x,0) &=& v_0(x), \quad \nabla\cdot v_0(x) = 0, \label{initialv} \\
F(x,0) &=& F_0(x) = \mathbb{I}, \label{initialF} \\
M(x,0) &=& M_0(x), \quad |M_0| \equiv 1, \label{initialM}
\end{eqnarray}
where $n$ denotes the outer normal to the boundary of $\Omega$.

\section{Main result}
\label{sec:MainResult}

As the main result of this contribution, we prove existence of weak solutions to the system \eqref{sys1}--\eqref{sys3}. We start by defining the notion of weak solutions we shall work with. Here and in the following we set $A=\frac12$, $\mu_0=1$ and $\gamma = \lambda = 1$ since constants are irrelevant for this mathematical analysis.

Moreover, we shall restrict our scope to $\Omega \subset \R^2$ in which we may obtain weak solution globally in time. If $\Omega  \subset \R^3$, the presented proof remains valid up to small modifications but only to obtain short-time existence of solutions; cf. Remark \ref{rem-3D} below. 

In the following, Bochner spaces are denoted by $L^p(\mathrm{O}; V)$, $W^{k,p}(\mathrm{O}; V)$ for functions mapping $\mathrm{O} \subset \R^m$ to a Banach space $V$ of which the norm in $V$ belongs to the appropriate Lebesgue or Sobolev space. In the special case in which $V$ is $\R^n$, we denote by $L^p_{\mathrm{div}}(\mathrm{O}; \R^n)$, $W^{1,p}_{0,\mathrm{div}}(\mathrm{O}; \R^n)$ those subsets of the appropriate Lebesgue or Sobolev space on which the distributional divergence vanishes. In the Sobolev space, also the boundary values (in the sense of trace) are $0$. We will use the notation $W^{-1,2}(\mathrm{O}; \R^n)$ for the dual space of $W^{1,2}_0(\mathrm{O}; \R^n)$; moreover, we shall denote the duality pairing between  $W^{-1,2}(\mathrm{O}; \R^n)$ and $W^{1,2}_0(\mathrm{O}; \R^n)$ by $\llangle \cdot, \cdot \rrangle$.
\begin{defi}
\label{def-weakSol}
Let $\Omega \subset \R^2$ be a $C^\infty$-domain and let $T > 0$ be the final time of the evolution. Then, we call $(v,F,M)$ enjoying the regularity
\begin{align*}
v &\in L^\infty(0,T; L^{2}_{\mathrm{div}}(\Omega,\R^2)) \cap L^2(0,T; W^{1,2}_{0,\mathrm{div}}(\Omega;\R^2)),  \\
F &\in L^\infty(0,T; L^2(\Omega;\R^{2 \times 2})) \cap L^2(0,T; W_0^{1,2}(\Omega;\R^{2 \times 2})), \\
M &\in L^\infty(0,T; W^{1,2}(\Omega;\R^3)) \cap L^2(0,T; W^{2,2}(\Omega;\R^3))
\end{align*}
a weak solution of the system \eqref{sys1}--\eqref{sys3} accompanied with initial/boundary conditions \eqref{simpleboundaryv}--\eqref{initialM} if it satisfies \eqref{boundaryMsimple} in the sense of trace as well as the initial conditions \eqref{initialv}--\eqref{initialM} in the sense
\begin{equation*}
 v(\cdot,t) \xrightarrow{L^2(\Omega)} v_0(\cdot), \qquad
 F(\cdot,t) \xrightarrow{L^2(\Omega)} F_0(\cdot), \qquad
 M(\cdot,t) \xrightarrow{W^{1,2}(\Omega)} M_0(\cdot) \qquad \text{as } t\to0^+,
\end{equation*}
and fulfills the system
\begin{align}
\nonumber & \int_0^{T}\!\!\!\int_\Omega - v\cdot \partial_t \phi+ (v\cdot\nabla)v \cdot \phi- \left(\nabla M\odot\nabla M - W'(F)F^\top - \nu \nabla v\right)\cdot \nabla \phi - (\nabla \Hext^\top M)\cdot\phi ~dx~dt \\
&\qquad = \int_\Omega v_0(x)\phi(x,0)dx\label{weakformv} \\[0.5em]
 & \int_0^{T}\!\!\!\int_\Omega - F \cdot \partial_t \xi+ (v\cdot\nabla) F \cdot \xi -  (\nabla v F) \cdot \xi + \kappa \nabla F \tp \nabla \xi~dx~dt = \int_\Omega F_0(x) \cdot \xi(x,0)~dx,  \label{weakformF} \\[0.5em]
& \nonumber \int_0^{T}\!\!\!\int_\Omega - M\cdot \partial_t \zeta +(v\cdot\nabla)M\cdot \zeta  + (M \times (\Delta M+H_\mathrm{ext})) \cdot \zeta -  |\nabla M |^2 M \cdot \zeta - \Delta M \cdot \zeta~dx~dt \\
&\qquad = \int_0^{T}\!\!\!\int_\Omega  \big( -M \cdot H_\mathrm{ext} M + H_\mathrm{ext}\big) \zeta dxdt + \int_\Omega M_0(x)\cdot \zeta(x,0)~dx, \label{weakformM}
\end{align}
for all $\phi(x,t) = \phi_1(t)\phi_2(x)$ with $\phi_1 \in W^{1,\infty}(0,T)$ satisfying $\phi_1(T)=0$ and $\phi_2 \in W^{1,2}_{0,\mathrm{div}}(\Omega;\R^2)$, for all $\xi(x,t) = \xi_1(t)\xi_2(x)$ with $\xi_1 \in W^{1,\infty}(0,T)$ satisfying $\xi_1(T)=0$ and $\xi_2 \in W^{1,2}(\Omega;\R^{2 \times 2})$ and all $\zeta(x,t) = \zeta_1(t)\zeta_2(x)$ with $\zeta_1 \in W^{1,\infty}(0,T)$ satisfying $\zeta_1(T)=0$ and $\zeta_2 \in L^2(\Omega;\R^3)$.
\end{defi}

In the weak formulation of \eqref{sys1} and \eqref{sys2} we used integration by parts to transfer the highest derivatives in the Laplacian to the test function, which is standard. Moreover, we used that, as long as $|M|=1$, \eqref{sys3} is equivalent to (see, e.g., \cite{BertschPodio-GuidugliValente2001,CarbouFabrie2001})
\begin{equation}
\partial_t M + (v\cdot\nabla) M = - M \times (\Delta M + H_\mathrm{ext})+ |\nabla M|^2M + \Delta M - M(M\cdot H_\mathrm{ext}) + H_\mathrm{ext}.
\label{LLG-EqFormWeakSol}
\end{equation}

Before formulating our main result, let us summarize the assumptions on the data in the model that we shall need: Let us start with the elastic energy $W$, which must satisfy $W(\r\Xi)=W(\Xi)$ for all $\r\in SO(d)$ (and thus $W'(\r\Xi) = \r W'(\Xi)$; see also \cite{LiuWalkington2001}). We assume that $W \in C^2(\R^{2\times 2})$ is of 2-growth, i.e., there exists a constant $C_1 > 0$ such that
\begin{equation}
\label{conditionW0}C_1|A|^2 \leq W(A) \leq C_1(|A|^2  + 1) \qquad \forall A \in \R^{2\times 2}.
\end{equation}
We assume that $W'(0)=0$. Further, notice that due to the differentiability of $W$ this implies that $W'(\cdot)$ is of 1-growth, that is
\begin{equation} \label{conditionW1}
|W'(A)| \leq C_2(|A|  + 1) \qquad \forall A \in \R^{2\times 2}
\end{equation}
and likewise $W''(\cdot)$ is bounded, i.e. 
\begin{equation} \label{conditionW2a}
|W''(A)| \leq C_3 \qquad \forall A \in \R^{2\times 2}.
 \end{equation}
Finally, we assume that $W$ is strictly convex; that is 
\begin{equation} 
\exists a > 0 \qquad (W''(\Xi)A)\cdot A \geq a |A|^2 \qquad \forall \Xi, A \in \R^{2 \times 2}.
\label{conditionW2}
\end{equation}

Our main result is the existence of weak solutions to \eqref{sys1}--\eqref{sys3} in the sense of Definition \ref{def-weakSol}:
\begin{theorem}\label{ThmSln}
Let $\Omega \subset \R^2$ be a $C^\infty$-domain and let $T > 0$ be the final time of the evolution. Let $W \in C^2(\R^{2\times 2};\R)$ satisfy  \eqref{conditionW0}--\eqref{conditionW2}. In addition, assume that
\begin{align}\label{HextRegularity}
\Hext &\in C^0(0,T;L^2(\Omega;\R^3)) \cap L^2(0,T;L^\infty(\Omega;\R^3)) \cap L^3(0,T;W^{1,4}(\Omega;\R^3))\\
\partial_t \Hext &\in L^1(0,T;L^1(\Omega;\R^3))
\end{align}
and $v_0\in L^2_{\mathrm{div}}( \Omega;\R^2)$, $F_0\in L^2(\Omega;\R^{2\times 2})$ and $M_0\in W^{2,2}(\Omega;\R^3)$.
Moreover, let the initial data and the external field satisfy the smallness condition
\begin{equation}\label{smallinitialdatacondThmB}
\mathrm{IED}:=\int_\Omega \frac{1}{2} |v_0|^2 + \frac{1}{2}|\nabla M_0|^2 + W(F_0) \, dx + 2\|H_\mathrm{ext}\|_{L^\infty(0,T; L^1(\Omega; \R^3))} + \|\partial_t H_\mathrm{ext}\|_{L^1(0,T; L^1(\Omega; \R^3))} < \frac{1}{\widetilde C}
\end{equation}
for a suitably small constant $\widetilde C>0$ depending just on $\Omega$. Then there exists a weak solution of the system \eqref{sys1}--\eqref{sys3} accompanied with initial/boundary conditions \eqref{simpleboundaryv}--\eqref{initialM} in the sense of Definition \ref{def-weakSol}.
\end{theorem}

We prove Theorem \ref{ThmSln} in Section \ref{sec-proof} below. The proof is based on a Galerkin approximation of the system \eqref{sys1}--\eqref{sys3}. As is standard in the context of the Navier-Stokes equation, we approximate the velocity in terms of basis functions of the Stokes operator. We leave  \eqref{sys2} as well as the LLG equation \eqref{sys3} undiscretized but insert the discretized velocity into these equations. A similar approach has already been used in \cite{LinLiu1995}, \cite{SunLiu2009} but here the partial discretization of the system is crucial also in order to keep the constraint $|M|=1$ satisfied in the Galerkin scheme.

In the Galerkin scheme, we are able prove enough regularity of $F$ and $M$ to be able to deduce the energy estimates, which in turn are used for converging the Galerkin scheme. However, the energetic a-priori estimates do not yield enough regularity of $M$ because we get $ \nabla M$ bounded only in $L^\infty(0,T; L^2(\Omega; \R^3)$. Thus, we need to adapt parts of the regularity analysis for the LLG equation (cf. e.g. \cite{CarbouFabrie2001,Melcher2007,Melcher2010}) to the case of our system. Our argument here is based on the technique from \cite{CarbouFabrie2001}.

A further peculiarity is brought into the proof by the fact that an adaptation of the technique of \cite{CarbouFabrie2001} to our case is fully possible only on the level of the Galerkin approximation since then $v$ is smooth. Nevertheless, we can obtain a bound on $\Delta M$ in $L^2(0,T; L^2(\Omega; \R^3))$ that is uniform in the Galerkin index. This is all that we need to make the limiting process work. Yet, all the higher regularities of $M$ that were obtained in \cite{CarbouFabrie2001} will blow-up if the limiting velocity is not Lipschitz continuous.

Before embarking onto the proof of Theorem~\ref{ThmSln}, let us consider some remarks about the assumptions of this Theorem as well as possible extensions.

\begin{rem}[Weak formulation of the LLG equation]
Let us note that our weak formulation of the LLG equation \eqref{weakformM} is actually stronger than the standardly used weak formulation as proposed in \cite{AlougesSoyeur1992}. Notice that we keep the highest derivatives (i.e.\ the Laplacian) in \eqref{weakformM} and, in fact, since no partial integration in space has been used, we can deduce from \eqref{weakformM} that the LLG equation actually holds a.e.\ in $\Omega$. We can afford to require this stronger formulation since we anyway need to prove a bound on  $\Delta M$ in $L^2(0,T; L^2(\Omega; \R^3))$ in order to be able to pass to the limit in the Galerkin approximation in the stress tensor.
\end{rem}

\begin{rem}[Convexity of $W$]
The convexity assumption \eqref{conditionW2} makes sure that the energy is lower semicontinuous which we will need in order to pass to the limit in the energy inequality. Nevertheless, this assumption is not optimal from the physical point of view since elastic energies in the large strain setting are not convex. In order to relax this assumption, it would be necessary to change the used mathematical methods; in particular the assumption \eqref{conditionW2} enters in Step 2 of the proof of Theorem \ref{ThmSln}.
\end{rem}

\begin{rem}[$\Omega \subset \R^2$]
\label{rem-3D}
The fact that $\Omega \subset \R^2$ enters at several places in the proof of the Theorem \ref{ThmSln} but most crucially in Step 2 where higher order a-priori estimates for the magnetization are derived and the Ladyzhenskaya inequality is used. Nevertheless, although we do not consider it here, the proof could be easily adapted, by using techniques from \cite{CarbouFabrie2001}, to hold also for $\Omega \subset \R^3$ but with a sufficiently short final time of the evolution $T$.
\end{rem}

\begin{rem}[Smallness of the initial data]
The smallness condition \eqref{smallinitialdatacondThmB} on the initial data is quite limiting but a condition of this type seems to be necessary in order to prove existence of weak solutions to \eqref{sys1}--\eqref{sys3}. In fact, in order to pass to the limit in the stress tensor in the balance of momentum, we need sufficient integrability of $\nabla M$ for which we employ the higher regularity of $M$. However, if the initial data are not small, higher regularity cannot be expected. Indeed, blowup in finite time for the LLG equation from smooth but not small initial data has been numerically reported in \cite{BartelsKoProhl2008}. An analytical proof of this phenomenon seems to be missing for the LLG equation but has been given in the related harmonic map heatflow equation in \cite{ChangDinqYe}.  
\end{rem}

\section{Proof of Theorem \ref{ThmSln}}
\label{sec-proof}

Let us now give a detailed proof of Theorem \ref{ThmSln}. Everywhere in the proof, we use $C$ as a generic constant that may change from expression to expression. It may only depend on the problem parameters that are fixed throughout the proof such as $\Omega$, but dependence on other data, in particular on the initial conditions or the Galerkin index is specified explicitly. Moreover, note that we do not always display the dependence of $v$ on $x$ and $t$; instead of $v(x,t)$ we may also write $v(t)$ ,if we want to stress the dependence on time, or just $v$; correspondingly for $F$ and $M$.

\begin{proof}[Proof of Theorem \ref{ThmSln}]
We start by constructing suitable approximate solutions: 

{\bf Step 1: Discrete formulation and existence of discrete solutions}\\
Let us construct Galerkin approximations of the velocity via eigenfunctions of the Stokes operator; i.e., let $\{\xi_i\}_{i=1}^\infty \subset C^\infty(\bar\Omega;\R^2)$ be an orthogonal basis of $W^{1,2}_{0,\mathrm{div}}(\Omega;\R^2)$ and an orthonormal basis of $L^2_{\mathrm{div}}(\Omega;\R^2)$ satisfying
\begin{equation}\label{StokesEVs}
\Delta\xi_i + \nabla p_i = -\lambda_i\xi_i
\end{equation}
in $\Omega$ and vanishing on the boundary. Here, $0<\lambda_1\leq\lambda_2\leq\cdots\leq\lambda_m\leq\cdots$ with $\lambda_m\xrightarrow{m\to\infty}\infty$. Notice that $\Omega$ is a $C^\infty$-domain so the assumed regularity of the eigenfunctions can indeed be guaranteed. Further, let us denote  
$$P_m:W^{1,2}_{0,\mathrm{div}}(\Omega;\R^2) \to\H_m := \mathrm{span} \{ \xi_1,\xi_1,\ldots,\xi_m \}.$$

We start by defining the notion of a weak solution to the approximate problem.
\begin{defi} 
\label{def-discrSol}
We call $(v_m, F_m, M_m)$ a weak discrete solution of the system  \eqref{sys1}--\eqref{sys3} on some time interval $(0,t)\subset (0,T)$ provided that the pair $(F_m, M_m)$ enjoys the following regularity\begin{align}
F_m & \in W^{1,2}(0,t; W^{-1,2}(\Omega; \R^{2 \times 2})) \cap L^\infty(0,t;L^2(\Omega;\R^{2\times 2})) \cap  L^2(0,t;W^{1,2}_0(\Omega;\R^{2 \times 2} )) \label{reqRegDiscrF}\\
M_m & \in W^{1,\infty}(0,t; L^2(\Omega; \R^3)) \cap L^\infty(0,t;W^{2,2}(\Omega; \R^3)) \cap L^{2}(0,t; W^{3,2}(\Omega; \R^3)) \label{reqRegDiscrM}
\end{align}
and solves 
\begin{align}
& \llangle {\partial_t} F_m , \Xi \rrangle + \int_\Omega (v_m\cdot\nabla)F_m \cdot \Xi - (\nabla v_m F_m) \cdot \Xi + \kappa \nabla F_m \cdot \nabla\Xi~dx = 0 & &\mbox{in } (0,t), \label{LLGPARTweakformFapproxm} \\
 & \nonumber \partial_t M_m + (v_m\cdot\nabla) M_m = |\nabla M_m|^2M_m + \Delta M_m \\& \qquad \qquad \qquad \qquad \quad \; - M_m \times (\Delta M_m + H_\mathrm{ext}) - M_m(M_m\cdot H_\mathrm{ext}) + H_\mathrm{ext}  & &\mbox{in } \Omega \times (0,t)
 \label{LLGPARTweakformMapproxm}
\end{align}
for all  $\Xi\in W^{1,2}_0(\Omega;\R^{2\times 2})$, together with the initial conditions \eqref{initialF}--\eqref{initialM} and boundary conditions \eqref{simpleboundaryF}--\eqref{boundaryMsimple}.

Moreover, $v_m (x,s) = \sum_{i=1}^m g^i_m(s)\xi_i(x)$ with $g^i_m: (0,t) \to \R$ being the Lipschitz continuous solution of 
\begin{align}\label{LLGPARTbasrepvapproxmODE}
\frac{d}{dt} g_m^i(s) &= -\nu\lambda_i g_m^i(s) + \sum_{j,k=1}^m g_m^j(s) g_m^k(s) A_{jk}^i + D_m^i(s, F_m,M_m),  \qquad \qquad \qquad i=1,\ldots,m,
\end{align}
with the initial condition $g_m^i(0) = \int_\Omega v_0 \cdot \xi_i~dx$ and
\begin{equation}
\begin{split}
A_{jk}^i &:= -\int_\Omega (\xi_j\cdot\nabla)\xi_k \cdot \xi_i~dx, \\
 D_m^i(s,F,M) &:= \int_\Omega \big( \nabla M(s)\odot\nabla M(s) -  W'(F(s))F(s)^\top \big) \cdot \nabla\xi_i~ + (\nabla H_\mathrm{ext}^\top(s) M(s)) \cdot \xi_i dx. \label{LLGPARTbasrepvapproxmOp}
\end{split}\end{equation}
 for any $s\in(0,t) $, $i,j,k=1,\ldots,m$ and any $(F,M)$ in the function spaces mentioned in \eqref{reqRegDiscrF} and \eqref{reqRegDiscrM}.
\end{defi}

For further convenience, let us denote
$$
\mathrm{IN}:=\big(\|W(F_0)\|_{L^1(\Omega)}, \|M_0\|_{W^{2,2}(\Omega;\R^3)}\big).
$$

We prove existence of discrete solutions to \eqref{sys1}--\eqref{sys3} in the sense of Definition \ref{def-discrSol} by a fixed point argument. To this end, we define for all $0<t_0 \leq T$ and $L=\|v_0\|_{L^2(\Omega;\R^2)} + 1$ the set
\begin{align*}
V_m(t_0) = \left\{ v(x,t) = \sum_{i=1}^m g_m^i(t) \xi_i(x) \mbox{ in } \Omega \times [0,t_0): \sup_{t\in [0,t_0)} \left(\sum_{i=1}^m \left| g_m^i(t) \right|^2 \right)^{\frac12} \leq L, \; g_m^i(0) = \int_\Omega v_0(x) \cdot \xi_i(x)~dx \right\}.
\end{align*}
Notice that $V_m(t_0)$ is a closed and convex subset of $C([0,t_0];\H_m)\subset C([0,t_0];L^2(\Omega;\R^2))$. With some $v \in V_m(t_0)$ fixed we may find weak solutions to \eqref{sys2}--\eqref{sys3} by means of the following lemma.

\begin{lemma}\label{LLGPARTLemParabolicEq} For $v\in V_m(t_0)$ fixed and $H_\mathrm{ext}$ satisfying \eqref{HextRegularity} there is a $0<t_1\leq t_0$, that only depends on $L$, $m$, $\mathrm{IN}$ and the external field $\Hext$, such that we can find unique $(F,M)$ with
\begin{align}
& F \in W^{1,2}(0,t_1; W^{-1,2}(\Omega; \R^{2 \times 2})) \cap L^2(0,t_1;W^{1,2}_0(\Omega;\R^{2 \times 2} )), \\
& M \in W^{1,\infty}(0,t_1; L^2(\Omega; \R^3)) \cap L^2(0,t_1; W^{3,2}(\Omega; \R^3))
\end{align}
 satisfying
\begin{align}
& \llangle {\partial_t} F , \Xi \rrangle + \int_\Omega (v \cdot\nabla)F \cdot \Xi - (\nabla v F) \cdot \Xi + \kappa \nabla F \cdot \nabla\Xi~dx = 0 \quad \mbox{ in } (0,t_1),  \label{Feq-Iterations}\\
 & \partial_t M + (v\cdot\nabla) M = |\nabla M|^2M + \Delta M 
 - M \times (\Delta M + H_\mathrm{ext}) - M(M\cdot H_\mathrm{ext}) + H_\mathrm{ext}  
 \quad \mbox{ in } \Omega \times (0,t_1), 
\label{Meq-Iterations}
\end{align}
for all  $\Xi\in W^{1,2}_0(\Omega;\R^{2\times 2})$, together with the initial conditions \eqref{initialF}--\eqref{initialM} and boundary conditions \eqref{simpleboundaryF}--\eqref{boundaryMsimple}.
Moreover, the pair $(F,M)$ satisfies the following bounds
\begin{align} \label{estLemma7}
& \|F\|_{L^\infty(0,t_1;L^2(\Omega;\R^{2\times 2}))} \leq C(L,m, \mathrm{IN})
 \qquad \|M\|_{L^\infty(0,t_1;W^{2,2}(\Omega;\R^3)) } \leq C(L,m,\mathrm{IN}, \Hext)
\end{align}
In addition, we have that $|M|=1$ a.e. in $\Omega\times(0,t)$ and the following estimate 
\begin{align}
\nonumber &\|\Delta M(t)\|_{L^2(\Omega; \R^3)}^2 \\ & \leq \|\Delta M_0\|_{L^2(\Omega; \R^3)}^2  + C(L,m,\Hext)\int_0^{t} \left(1+\|\nabla M\|_{L^2(\Omega; \R^{3 \times 2})}^6+ \|\nabla M\|_{L^2(\Omega; \R^{3 \times 2})}^2 \|\Delta M\|_{L^2(\Omega; \R^3)}^2\right) ds 
\label{LinftyH2onM}
\end{align}
for any $t$ for which the equation in \eqref{Meq-Iterations} is satisfied.
\end{lemma}

The proof of Lemma \ref{LLGPARTLemParabolicEq} is based on a Galerkin approximation within which the estimates \eqref{estLemma7} and \eqref{LinftyH2onM} can be obtained by following the reasoning of \cite{CarbouFabrie2001}. We postpone it, for the sake of clarity, to Section~\ref{sec:PfsLem} and rather continue with the proof of Theorem \ref{ThmSln} at this point. 

By Lemma~\ref{LLGPARTLemParabolicEq}, we have now found, for some fixed $v\in V_m(t_0)$, functions $(F,M)$ that solve \eqref{Feq-Iterations}--\eqref{Meq-Iterations} and are such that
$$
D^i_m(t; F,M) \in L^\infty(0,t_1),
$$
with the $L^\infty$-norm of $D^i_m(t; F,M)$ depending only on $L$ and $m$, the initial data through $\mathrm{IN}$ and the external magnetic field.

Thus, we can apply Carath\'eodory's existence theorem to obtain existence of unique Lipschitz continuous solutions $\tilde{g}_m^i(t)$ of
\begin{equation}
\frac{d}{dt} \tilde{g}_m^i(t) = -\nu\lambda_i \tilde{g}_m^i(t) + \sum_{j,k=1}^m \tilde{g}_m^j(t) \tilde{g}_m^k(t) A_{jk}^i + D^i(t; F,M),   \qquad \qquad i=1,\ldots,m,
\label{ODE-Iterations}
\end{equation}
with the initial condition $\tilde g_m^i(0) = \int_\Omega v_0 \cdot \xi_i~dx = g_m^i(0)$, at least on a time interval $(0,t_2)$ with $t_2 \leq t_1$.
Notice that, for $t\in[0,t_1]$ and for $\|\tilde{g}_m-g_m(0)\|\leq b$, $b>0$, where $\tilde{g}_m=(\tilde{g}_m^1,\ldots,\tilde{g}_m^m)$, we can bound the right-hand side of \eqref{ODE-Iterations} by the constant
\begin{equation*}
R=-\nu\lambda_i (2b+\|g_m(0)\|) + (2b+\|g_m(0)\|)^2 \sum_{j,k=1}^m \left|A_{jk}^i\right| + \|D^i(t; F,M)\|_{ L^\infty(0,t_1)}.
\end{equation*}
Thus, it follows from \cite[Chapter~1, Theorem~1]{Filippov1988} that $t_2$ has to be chosen in such a way that $R t_2 \leq b$; in other words $t_2$ depends just on the $L^\infty$-norm of $D^i_m(t; F,M)$ (that in turn only depends on $L$ and $m$, the initial data through $\mathrm{IN}$ and the external magnetic field).

Choosing $0<t^\ast \leq t_2$ small enough (but, as we shall see, only dependent on $L$ and $m$, the initial data through $\mathrm{IN}$ and the external magnetic field), we can assure that  
\begin{align}\label{tilde}
\tilde v(x,t) = \sum_{i=1}^m \tilde g_m^i(t) \xi_i(x)
\end{align}
 is in $V_m(t^\ast)$. To prove this, note that we can deduce from \eqref{ODE-Iterations} and \eqref{LLGPARTbasrepvapproxmOp} that $\tilde v$ satisfies
\begin{equation}
\int_\Omega \partial_t \tilde v \cdot \zeta + (\tilde v\cdot\nabla)\tilde v\cdot \zeta- \left(\nabla M \odot\nabla M - W'(F)F^\top - \nu \nabla \tilde v \right)\cdot \nabla \zeta -  (\nabla H_\mathrm{ext}^\top M)\cdot \zeta ~dx = 0
\label{v-EqDisc}
\end{equation}
for all $\zeta \in \H_m$. Testing this with $\tilde v$ itself yields 
\begin{align}
\nonumber \frac12 \frac{d}{dt}&\|\tilde v(t)\|_{L^2(\Omega;\R^2)}^2\\
&= - \underbrace{\int_\Omega (\tilde v\cdot\nabla)\tilde v \cdot \tilde v ~dx }_{=0} + \nu\|\nabla \tilde v\|_{L^2(\Omega)}^2 +  \int_\Omega (\nabla M\odot\nabla M - W'(F)F^\top)\cdot\nabla\tilde v ~dx + \int_\Omega  (\nabla \Hext^\top M) \cdot \tilde v~dx \nonumber\\
&\leq  C(m, H_\mathrm{ext}) \|\tilde v (t)\|_{L^2(\Omega)} \left\|\int_\Omega |\nabla M\odot\nabla M - W'(F)F^\top| ~dx\right\|_{L^\infty(0,t_2)} \nonumber \\
& \leq  \|\tilde v (t)\|_{L^2(\Omega;\R^2)} C(m,L,\mathrm{IN},H_\mathrm{ext}), \label{Iteration-vEst}
\end{align}
where $C(m,L,\mathrm{IN},\Hext)$  is obtained via Lemma \ref{LLGPARTLemParabolicEq}. In this estimate we used that we obtain $\|\nabla \tilde{v}(t)\|_{L^\infty(\Omega;\R^2)} \leq C(m) \|\tilde v(t)\|_{L^2(\Omega;\R^2)}$ for $\tilde v \in \H_m$. 
From \eqref{Iteration-vEst}, we get that 
$$
\frac{d}{dt} \|\tilde v (t)\|_{L^2(\Omega;\R^2)} \leq C(m,L,\mathrm{IN},H_\mathrm{ext});
$$
indeed, this is obvious if $\|\tilde v\|_{L^2(\Omega;\R^2)}=0$ and it follows in all other cases by rewriting $\frac{d}{dt} \|\tilde v (t)\|_{L^2(\Omega;\R^2)} = \frac{ \frac{d}{dt} \|\tilde v\|_{L^2(\Omega;\R^2)}^2 (t) }{2\|\tilde v\|_{L^2(\Omega;\R^2)} (t)} $. Thus, we see that
$$
\|\tilde v (t)\|_{L^2(\Omega;\R^2)} \leq \|\tilde v_0\|_{L^2(\Omega;\R^2)} + C(m,L,\mathrm{IN},H_\mathrm{ext})t.
$$
Now, we can define an operator 
\begin{equation}
\mathcal{L}: V_m(t^\ast) \to V_m(t^\ast), \quad v\mapsto\tilde v,
\label{operatorL}
\end{equation}
with $\tilde v$ defined as in \eqref{tilde}. Notice that the range of $\mathcal{L}$ is precompact in $C([0,t^\ast];\H_m)$. This can be seen from the Arzel\`a-Ascoli theorem since any $\tilde{v}$ in the range of $\mathcal{L}$ is obtained from \eqref{ODE-Iterations} and thus is uniformly Lipschitz continuous in time with a Lipschitz constant depending just on $L$ and $m$, the initial data through $\mathrm{IN}$ and the external magnetic field.

Moreover, we will prove in the following lemma (the proof of which is technical but straightforward and thus postponed to Section~\ref{sec:PfsLem}) that $\mathcal{L}$ is continuous.

\begin{lemma}
\label{lemma-cont}
The operator $\mathcal{L}$ defined in \eqref{operatorL} is continuous on $V_m(t^\ast)$ in the topology of $C(0,t^\ast;\H_m)$.
\end{lemma}

Thus, Schauder's fixed point theorem assures the existence of a $$v_m \in V_m(t^\ast)$$ such that $\mathcal{L}(v_m)=v_m$. In turn, $v_m$ together with the associated pair $(F_m, M_m)$ is a discrete weak solution in the sense of Definition~\ref{def-discrSol} of the system \eqref{sys1}--\eqref{sys3} on the time interval $[0,t^\ast]$.

{\bf Step 2: A-priori estimates}\\
Let us now deduce the a-priori estimates, i.e., in particular \eqref{a-prioriI} and \eqref{a-prioriII} below. To this end, let us first multiply  equation \eqref{LLGPARTweakformMapproxm} by $-H_\mathrm{eff}=-\Delta M_m - H_\mathrm{ext}$ to get that
\begin{equation}
(\partial_t M_m + (v_m\cdot\nabla) M_m)\cdot (-H_\mathrm{eff}) = (M_m\times H_\mathrm{eff})\cdot H_\mathrm{eff} + (M_m \times M_m \times H_\mathrm{eff}) \cdot H_\mathrm{eff} = |M_m\cdot H_\mathrm{eff}|^2 - |H_\mathrm{eff}|^2 \leq  0,
\label{LLG3CrossEst}
\end{equation}
since $|M_m| = 1$ by Lemma~\ref{LLGPARTLemParabolicEq}. After plugging the definition of the effective field into this equation, we obtain
\begin{equation}
\frac{\mathrm{d}}{\mathrm{d}t} \int_\Omega \frac{1}{2}|\nabla M_m|^2 - M_m \cdot H_\mathrm{ext} dx + \int_\Omega  M_m\cdot \partial_t H_\mathrm{ext} dx -\int_\Omega ((v_m\cdot\nabla) M_m)\cdot (\Delta M_m + H_\mathrm{ext}) \leq 0.
\label{AprioriMagI}
\end{equation}
Note that for any smooth $M$ the following identity holds $\nabla\cdot\big( \nabla M\otimes\nabla M \big) = \nabla \frac{|\nabla M|^2}{2} + (\nabla M)^\top\Delta M$ and thus $\nabla\cdot\big( \nabla M\otimes\nabla M\big)\cdot v_m = \nabla \frac{|\nabla M|^2}{2} \cdot v_m + (v_m \cdot \nabla) M \Delta M$. Therefore, using integration by parts and the fact that $v_m$ is divergence free together with the vanishing boundary conditions, we obtain the following identity
$$
-\int_\Omega (\nabla M\otimes\nabla M \big)\cdot \nabla v_m \, dx= \int_\Omega (v_m \cdot \nabla) M \Delta M \, dx,
$$
which holds by approximation also for $M_m$ for almost all $t \in [0,t^\ast)$. Moreover, by integration by parts we get that 
$$
-\int_\Omega (v_m \cdot \nabla M_m)\cdot H_\mathrm{ext} = \int_\Omega \nabla \cdot v_m M_m\cdot H_\mathrm{ext} + (\nabla H_\mathrm{ext}^\top M_m)\cdot v_m \, dx = \int_\Omega (\nabla H_\mathrm{ext}^\top M_m)\cdot v_m \, dx.
$$
Plugging this into equation \eqref{AprioriMagI} leads to
\begin{align} \label{apriorimag}
&\frac{\mathrm{d}}{\mathrm{d}t} \int_\Omega \frac{1}{2}|\nabla M_m|^2 -  M_m \cdot H_\mathrm{ext} dx +\int_\Omega  M_m\cdot \partial_t H_\mathrm{ext} dx +\int_\Omega (\nabla M_m \odot \nabla M_m)\cdot \nabla v_m + (\nabla H_{\mathrm{ext}}^\top M_m)\cdot v_m \, \, dx \leq 0.
\end{align}

Let us now test \eqref{LLGPARTweakformFapproxm} with $W'(F_m)$. Notice that this is an admissible test function since for almost all $t \in [0,t^\ast)$ we have that $W'(F_m)$ is in $W^{1,2}(\Omega; \R^{2 \times 2})$. Indeed, due to growth condition \eqref{conditionW1}, $W'(F_m)$ is in $L^2(\Omega;\R^{2 \times 2})$ if $F_m\in L^2(\Omega;\R^{2 \times 2})$, which is guaranteed by Lemma~\ref{LLGPARTLemParabolicEq}. Moreover, since $W''(\cdot)$ is bounded by \eqref{conditionW2a}, $\nabla W'(F_m) = W''(F_m)\nabla F_m$ is in $L^2(\Omega;\R^{2 \times 2 \times 2})$ for almost all $t \in [0,t^\ast)$ if $\nabla F_m\in L^2(\Omega;\R^{2 \times 2 \times 2})$, which is again guaranteed by Lemma \ref{LLGPARTLemParabolicEq} where a bound on $F_m$ in $L^2(0,t^\ast; W^{1,2}(\Omega; \R^{2 \times 2}))$ is obtained. Finally, due to the continuity of the trace operator and $W'(0)=0$, we know that $W'(F_m)=0$ on $\partial\Omega$. Plugging in the test, we obtain
\begin{align*}
\frac{\mathrm{d}}{\mathrm{d}t} &\int_\Omega W(F_m) dx  +\int_\Omega \kappa \nabla F_m {\cdot} \nabla W'(F_m) + ((v\cdot\nabla) F_m - \nabla v_m F_m)W'(F_m) \, dx  \\
&=\frac{\mathrm{d}}{\mathrm{d}t} \int_\Omega W(F_m) dx  +\int_\Omega \kappa \nabla F_m {\cdot} ( W''(F_m) \nabla F_m) + (v\cdot\nabla) W(F_m) - (W'(F_m) F_m^\top)\cdot \nabla v_m \, dx  = 0.
\end{align*}
Therefore, using that $v_m$ is divergence free and plugging in condition \eqref{conditionW2}, we get that
$$
\frac{\mathrm{d}}{\mathrm{d}t} \int_\Omega W(F_m) dx  +\int_\Omega \kappa a |\nabla F_m|^2  - (W'(F_m) F_m^\top)\cdot \nabla v_m~dx \leq 0.
$$
Lastly, we deduce from \eqref{LLGPARTbasrepvapproxmODE} and \eqref{LLGPARTbasrepvapproxmOp} that $v_m=\sum_{i=1}^m g_m^i(t) \xi_i(x)$ satisfies
\begin{equation}
\int_\Omega \partial_t v_m\cdot \zeta + (v_m\cdot\nabla)v_m \cdot \zeta- \left(\nabla M_m\odot\nabla M_m - W'(F_m)F_m^\top - \nu \nabla v_m\right)\cdot \nabla \zeta -  (\nabla H_\mathrm{ext}^\top M_m)\cdot \zeta ~dx = 0
\label{v-EqDisc}
\end{equation}
for all $\zeta \in \H_m$. Testing this equality with $v_m$ itself yields
\begin{align*}
&\frac12\frac{\mathrm{d}}{\mathrm{d}t} \int_\Omega |v_m|^ 2 dx + \frac12 \int_\Omega \nu |\nabla v_m|^ 2 - (\nabla M_m\odot\nabla M_m)\cdot \nabla v_m + (W'(F_m)F_m^\top)\cdot \nabla v_m   -  (\nabla H_\mathrm{ext}^\top M_m)\cdot v_m~dx = 0,
\end{align*}
because $\int_\Omega (v_m\cdot \nabla v_m) \cdot v_m dx = 0$.
Summing the three expressions above, we get the overall energy inequality for any $t\in[0,t^\ast)$ as follows:
\begin{align}
\nonumber &\underbrace{\int_\Omega \frac{1}{2} |v_m(t)|^2 + \frac{1}{2}|\nabla M_m(t)|^2 - M_m(t) \cdot H_\mathrm{ext}(t) + W(F_m(t)) \, dx}_\text{energy at time $t$} + \underbrace{\int_0^t\!\!\!\int_\Omega \kappa a |\nabla F_m|^2 dx ds}_\text{regularization}+ \underbrace{\int_0^t\!\!\!\int_\Omega\nu |\nabla v_m|^ 2 \, dx\,ds}_\text{dissipation}  
\\ &\qquad \leq \underbrace{\int_\Omega \frac{1}{2} |v_m(0)|^2 + \frac{1}{2}|\nabla M_m(0)|^2 - M_m(0) \cdot H_\mathrm{ext}(0) + W(F_m(0)) \, dx}_\text{approximate initial energy} \underbrace{-\int_0^t\!\!\!\int_\Omega M_m\cdot  \partial_t H_\mathrm{ext} \, dx\,ds}_\text{work of external forces} \nonumber
\\ &\qquad \leq \underbrace{\int_\Omega \frac{1}{2} |v_0|^2 + \frac{1}{2}|\nabla M_0|^2 - M_0 \cdot H_\mathrm{ext}(0) + W(F_0) \, dx}_\text{initial energy}  \underbrace{-\int_0^t\!\!\!\int_\Omega M_m\cdot  \partial_t H_\mathrm{ext} \, dx\,ds,}_\text{work of external forces}
\label{energy-Est}
\end{align}
where in the last line we exploited that $F_m$ and $M_m$ already satisfy the initial conditions exactly.
From \eqref{energy-Est}, we  obtain the following estimate for any $t\in [0,t^\ast)$ 
\begin{align}
\nonumber &\int_\Omega \frac{1}{2} |v_m(t)|^2 + \frac{1}{2}|\nabla M_m(t)|^2  + W(F_m(t)) \, dx+ \int_0^t\!\!\!\int_\Omega \kappa a |\nabla F_m|^2 + \nu |\nabla v_m|^ 2 \, dx\,ds  \\ &\qquad \leq \underbrace{\int_\Omega \frac{1}{2} |v_0|^2 + \frac{1}{2}|\nabla M_0|^2 + W(F_0) \, dx + 2\|H_\mathrm{ext}\|_{L^\infty(0,T; L^1(\Omega; \R^3))} + \|\partial_t H_\mathrm{ext}\|_{L^1(0,T; L^1(\Omega; \R^3))}}_\text{IED};
\label{a-prioriI}
\end{align}
plugging in additionally \eqref{conditionW0} we obtain that 
\begin{equation}
\sup_{t \in [0,t^\ast)} \int_\Omega \frac{1}{2} |v_m(t)|^2 + \frac{1}{2}|\nabla M_m(t)|^2 + |F_m(t)|^2 \, dx+ \int_0^t\!\!\!\int_\Omega \kappa a |\nabla F_m|^2 + \nu |\nabla v_m|^ 2 \, dx\,ds \leq C(\text{IED}).
\label{a-prioriI-final}
\end{equation}

The above estimate is based on the inequality in \eqref{LLG3CrossEst}, i.e.\ on $|M\cdot H_\mathrm{eff}|^2-|H_\mathrm{eff}|^2 \leq 0$, cf.\ \eqref{AprioriMagI}. We can refine the a priori estimate in \eqref{apriorimag} by working with the following expression obtained with $H_\mathrm{eff} = \Delta M_m + H_\mathrm{ext}$.
$$
|M_m\cdot H_\mathrm{eff}|^2-|H_\mathrm{eff}|^2 = (M_m\cdot \Delta M_m)^2  + 2(M_m\cdot\Delta M_m) (M_m\cdot H_\mathrm{ext}) + (M_m\cdot H_\mathrm{ext})^2 - |\Delta M_m|^2 + 2 \Delta M_m \cdot H_\mathrm{ext} + |H_\mathrm{ext}|^ 2.
$$
For any $t\in[0,t^\ast)$ we get by the same procedure as above that 
\begin{align*}
\nonumber &\int_\Omega \frac{1}{2} |v_m(t)|^2 + \frac{1}{2}|\nabla M_m(t)|^2  + W(F_m(t)) \, dx + \int_0^t\!\!\!\int_\Omega \kappa a |\nabla F_m|^2 + \nu | \nabla v_m|^ 2 + |\Delta M_m|^2 + |H_\mathrm{ext}|^ 2 \, dx\,ds \\ &\qquad \leq \int_\Omega \frac{1}{2} |v_0|^2 + \frac{1}{2}|\nabla M_0|^2 - M_0 \cdot H_\mathrm{ext}(0) + M_m(t) \cdot H_\mathrm{ext}(t) + W(F_0) \, dx - \int_0^t\!\!\!\int_\Omega M_m\cdot  \partial_t H_\mathrm{ext} \, dx\,ds \\ &\qquad \qquad + \int_0^t\!\!\!\int_\Omega |\nabla M_m|^4  + 2(M_m\cdot\Delta M_m)(M_m\cdot H_\mathrm{ext}) + (M_m\cdot H_\mathrm{ext})^2 -  2 \Delta M_m \cdot H_\mathrm{ext} \, dx ds,
\end{align*}
where in the last term, we used that $-M_m\cdot \Delta M_m = |\nabla M_m|^2$ since the modulus of $M_m$ is equal to one. By Young's and H\"older's inequalities, this leads to 
\begin{align*}
\nonumber &\int_\Omega \frac{1}{2} |v_m(t)|^2 + \frac{1}{2}|\nabla M_m(t)|^2 + W(F_m(t)) \, dx + \int_0^t\!\!\!\int_\Omega \kappa a |\nabla F_m|^2 + \nu |\nabla v_m|^ 2 + (1-2\varepsilon^2)|\Delta M_m|^2 \, dx\,ds \\ 
&\qquad \leq \text{IED}+ \int_0^t\!\!\!\int_\Omega |\nabla M_m|^4 + \frac{1}{2\varepsilon^2} |H_\mathrm{ext}|^2 \, dx ds,
\end{align*}
where $\epsilon>0$ can be arbitrarily small. 

Now, we exploit an observation from \cite{CarbouFabrie2001}: the term $\int_0^t\!\int_\Omega |\nabla M_m|^4 \, dx dt$ on the right-hand side of the above expression can actually be absorbed into the Laplacian on the left-hand side, which yields a bound on the second gradient of $M_m$. 
To this end, observe that for any $t\in[0,t^\ast)$  
\begin{align} \label{Delta-equals-secondgrad}
 \|\Delta M_m(t)\|_{L^2(\Omega; \R^{3})} =\|\nabla^2 M_m(t)\|_{L^2(\Omega; \R^{3 \times 2 \times 2})} 
\end{align}
due  to the Neumann boundary conditions for $M_m$. Further, by Ladyzhenskaya's inequality, it holds for any $f\in W^{1,2}(\Omega; \R^{\tilde d})$ that
\begin{align} \label{Lady}
\|f\|_{L^4(\Omega;\R^{\tilde d})} &\leq  C \left( \|f\|_{L^2(\Omega;\R^{\tilde d})}+ \|\nabla f \|_{L^2(\Omega;\R^{\tilde d \times 2})}^{1/2} \|f\|_{L^2(\Omega;\R^{\tilde d})}^{1/2} \right)
\end{align}
for some $C>0$ depending on $\Omega$ only. 
Hence,
\begin{equation}\label{LLGPARTSobolevestimated2}
\|\nabla M_m\|^4_{L^4(\Omega;\R^{3\times 2})} \leq \widetilde C \left(\|\nabla M_m\|^4_{L^2(\Omega; \R^{3 \times 2})} +  
\|\nabla^2 M_m\|^2_{L^2(\Omega; \R^{3\times 2\times 2})}\|\nabla M_m\|^2_{L^2(\Omega; \R^{3 \times 2})} \right)
\end{equation}
for some $\widetilde C>0$ depending on $\Omega$ only.
Thus, we have for any $t\in[0,t^\ast)$ that
\begin{align*}
\nonumber &\int_\Omega \frac{1}{2} |v_m(t)|^2 + \frac{1}{2}|\nabla M_m(t)|^2 + W(F_m(t)) \, dx + \int_0^t\!\!\!\int_\Omega \kappa a |\nabla F_m(t)|^2 + \nu |\nabla v_m|^ 2 + (1-2\varepsilon^2)|\nabla^2 M_m|^2 \, dx\,ds \\ 
&\qquad \leq \ \text{IED} + \widetilde C \int_0^t \|\nabla M_m\|^4_{L^2(\Omega; \R^{3 \times 2})} +  \|\nabla^2 M_m\|_{L^2(\Omega; \R^{3\times 2\times 2}}^2\|\nabla M_m\|_{L^2(\Omega; \R^{3 \times 2})}^{2} ~ds + \int_0^t\int_\Omega \frac{1}{2\varepsilon^2} |H_\mathrm{ext}|^2 \, dx ds \\
&\qquad \leq \ (1+ \widetilde C T)\text{IED}  + \widetilde C\;\text{IED}\int_0^t  \|\nabla^2 M_m\|_{L^2(\Omega; \R^{3\times 2\times 2})}^2 \, ds + \frac{2}{\varepsilon^2}\|H_\mathrm{ext}\|^2_{L^2(0,T; L^2(\Omega; \R^3))},
\end{align*}
where we applied that $\|\nabla M_m\|^2_{L^2(\Omega; \R^{3 \times 2})} \leq \text{IED}$ uniformly in the time by \eqref{a-prioriI}.
Therefore, if $ \widetilde C \; \text{IED} < 1-2\varepsilon^2$, we get, additionally to \eqref{a-prioriI}, the a-priori estimate:
\begin{equation}
\|\nabla^2 M_m\|^2_{L^2(0,t^\ast;L^2(\Omega; \R^{3 \times 2 \times 2}))} \leq C(T, \text{IED}, H_\mathrm{ext}). \label{a-prioriII}
\end{equation}

Notice that, owing to estimate \eqref{LinftyH2onM} we can strengthen \eqref{a-prioriII} albeit not uniformly in the Galerkin variable $m$. Indeed, since $\|\nabla M_m(t)\|_{L^2(\Omega; \R^{3 \times 2})}$ is bounded uniformly by $\mathrm{IED}$ on $(0,t^\ast)$, we may rewrite \eqref{LinftyH2onM} as 
$$
\|\Delta M_m(t)\|_{L^2(\Omega;\R^3)} \leq \|\Delta M_0\|_{L^2(\Omega; \R^3)} + C(L,m,\mathrm{IED}, H_\mathrm{ext})\int_0^{t}(1+\|\Delta M_m(s)\|_{L^2(\Omega;\R^3)}^4) ds;
$$
whence we obtain by the Gronwall lemma that for all $t \in [0,t^\ast)$
\begin{equation}
\|\Delta M_m(t)\|_{L^2(\Omega;\R^3)}\leq C(L,m,\mathrm{IED},H_\mathrm{ext})(\|\Delta M_0\|_{L^2(\Omega; \R^3)} + T),
\label{FinalEst-LinftyH2M}
\end{equation}
where we also used that $\int_0^{t^\ast} \|\Delta M_m(s)\|_{L^2(\Omega;\R^3)}^2 ds \leq \mathrm{IED}.$

\noindent{\bf Step 3: Dual a-priori estimates}

Notice that the a-priori estimates obtained in Step 2 do not give any information on the time derivatives of the quantities $v_m$, $F_m$, $M_m$. However, these will be needed since without a uniform bound on time derivatives we cannot expect strong convergence in Bochner spaces, which in turn is crucial to pass to the limit in the non-linearities in the system. We deduce these estimates directly from the discrete system \eqref{LLGPARTweakformFapproxm}, \eqref{LLGPARTweakformMapproxm} and \eqref{v-EqDisc} itself. To this end, let $t\in [0,T]$ be such that $(v_m,F_m,M_m)$ is a weak solution in the sense of Definition~\ref{def-discrSol}.

Indeed, for the velocity (notice that the velocity is, albeit not uniformly in $m$, Lipschitz in time which makes the time integration feasible) and for any $\xi \in L^2(0,t)$ and any $\zeta \in W^{1,2}_{0,\mathrm{div}}(\Omega; \R^2)$ with 
$$
\|\xi\|_{L^2(0,t)} \leq 1 \qquad \text{and} \qquad \|\zeta\|_{W^{1,2}(\Omega; \R^2)} \leq 1
$$
we get that
\begin{align*}
&\int_0^{t} \!\!\!\!\! \int_\Omega \partial_t v_m \cdot (\xi \zeta)~dx~ds = \int_0^{t} \!\!\!\!\! \int_\Omega \partial_t v_m \cdot (\xi P_m\zeta)~dx~ds \\
&\qquad =\int_0^{t} \!\!\!\!\! \int_\Omega -(v_m\cdot\nabla) v_m (\xi P_m\zeta) + \nu \nabla v_m \cdot (\xi \nabla P_m\zeta)  +\big(\nabla M_m\odot\nabla M_m - W'(F_m) F_m^\top \big)\cdot (\xi \nabla P_m\zeta) \\ & \qquad\qquad\qquad+ (\nabla H_\mathrm{ext}^\top M_m)\cdot (\xi P_m\zeta) ~dx~ds \\
&\qquad =\int_0^{t} \!\!\!\!\! \int_\Omega   (v_m \otimes v_m) \cdot(\xi \nabla P_m \zeta) + \nu \nabla v_m \cdot (\xi \nabla P_m\zeta)  +\big(\nabla M_m\odot\nabla M_m + W'(F_m) F_m^\top \big)\cdot (\xi \nabla P_m\zeta)\\ 
& \qquad\qquad\qquad+ (\nabla H_\mathrm{ext}^\top M_m)\cdot (\xi P_m\zeta) ~dx~ds \\
&\qquad \leq \int_0^{t} \!\!\!\!\! \int_\Omega   |v_m|^2 |\xi| |\nabla P_m \zeta| + \nu |\nabla v_m||\xi||\nabla P_m\zeta|  +\big(|\nabla M_m\odot\nabla M_m| + |W'(F_m) F_m^\top| \big)|\xi| |\nabla P_m\zeta| \\ & \qquad\qquad\qquad+ |\nabla H_\mathrm{ext}^\top M_m| |\xi| |P_m\zeta| ~dx~ds \\
&\qquad \leq 
\int_0^{t} \|v_m\|^2_{L^4(\Omega;\R^2)} |\xi| \|\nabla P_m \zeta\|_{L^2(\Omega,\R^{2 \times 2})} + \nu \|\nabla v_m\|_{L^2(\Omega;\R^{2\times 2})} |\xi| \|\nabla P_m\zeta\|_{L^2(\Omega;\R^{2 \times 2})} ~ds \\
&~~~~~~~~~~~~~~~ + \int_0^{t} \big(\|\nabla M_m\odot\nabla M_m\|_{L^2(\Omega;\R^{2\times 2})}+ \|W'(F_m)F_m^\top\|_{L^2(\Omega;\R^{2 \times 2})} \big) |\xi| \|\nabla P_m \zeta\|_{L^2(\Omega;\R^{2 \times 2})} ~ds\\
&~~~~~~~~~~~~~~~ + \int_0^{t} \|\nabla \Hext\|_{L^2(\Omega;\R^{3\times 2})}  \|M_m\|_{L^\infty(\Omega;\R^{3})} |\xi| \| P_m \zeta\|_{L^2(\Omega;\R^{2})} ~ds\\
& \qquad \leq c \|v_m\|_{L^4(0,t;L^4(\Omega,\R^2))}^2 + \nu\|\nabla v_m\|_{L^{2}(0,t;L^2(\Omega))} + \|\nabla M_m\|^2_{L^{4}(0,t;L^4(\Omega;\R^{3 \times 2}))} \\
&\qquad \qquad    +c(1+ \|F_m\|_{L^{4}(0,t;L^4(\Omega;\R^{2 \times 2}))}^2) + \|\nabla \Hext\|_{L^2(0,T;L^2(\Omega;\R^{3\times 2}))}  \|M_m\|_{L^\infty(0,t;L^\infty(\Omega;\R^{3}))},
\end{align*}
where we used that $\|P_m \zeta\|_{W^{1,2}_0(\Omega; \R^2)} \leq \|\zeta\|_{W^{1,2}(\Omega; \R^2)} \leq 1$ and exploited the growth condition \eqref{conditionW1}. Notice that the terms appearing on the right-hand side of this expression are bounded by interpolation of the energetic estimate that we already obtained in the previous step. Indeed, applying Ladyzhenskaya's inequality \eqref{Lady} to $v_m$, $\nabla M_m$ and $F_m$, we get the asserted Bochner-space regularities.

Thus, taking a supremum over all $\xi$ and $\zeta$ as above, we see that 
\begin{equation}
\|\partial_t v_m\|_{L^{2}(0,t; W^{-1,2}(\Omega; \R^2))} \leq C(T, \text{IED}, H_\mathrm{ext}).
\label{a-prioriTimeDerV}
\end{equation}

In the same spirit, we deduce also estimates on the time derivatives of the magnetization $M_m$ and the deformation gradient $F_m$. Let us start with the magnetization, multiply \eqref{LLGPARTweakformMapproxm} by some arbitrary $\xi \in L^2(0,t)$ and any $\zeta \in L^2(\Omega; \R^3)$ satisfying
$$
\|\xi\|_{L^2(0,t)} \leq 1 \qquad \text{and} \qquad \|\zeta\|_{L^2(\Omega; \R^3)} \leq 1
$$
and integrate over $\Omega$ and $(0,t)$. We obtain
\begin{align*}
& \int_0^{t}\!\!\!\!\!\!\int_\Omega {\partial_t} M_m \cdot (\xi \zeta)~dx~ds \\ \leq & 
\int_0^{t}\!\!\!\!\!\!\int_\Omega  |(v_m\cdot\nabla)M_m\cdot (\xi \zeta)| + |(M_m\times (\Delta M_m+H_\mathrm{ext})) \cdot (\xi \zeta)| + |(M_m \times M_m\times (\Delta M_m+H_\mathrm{ext})) \cdot (\xi \zeta)| dx ds\\
\leq & 
\int_0^{t}\!\!\!\|v_m\|_{L^4(\Omega;\R^ 2)} \|\nabla M_m\|_{L^4(\Omega;\R^{3 \times 2})} |\xi| \|\zeta\|_{L^2(\Omega;\R^2)} + 2\big(\|\Delta M_m\|_{L^2(\Omega;\R^3)} +\|H_\mathrm{ext}\|_{L^2(\Omega;\R^3)}\big) |\xi| \|\zeta\|_{L^2(\Omega;\R^3)} \\
\leq & 
\|v_m \|_{L^{4}(0,t;L^4(\Omega;\R^2))} \|\nabla M_m\|_{L^{4}(0,t;L^4(\Omega;\R^{3 \times 2}))}  + 2(\|\Delta M_m\|_{L^{2}(0,t;L^2(\Omega;\R^ 3))} + \|H_\mathrm{ext}\|_{L^{2}(0,T;L^2(\Omega;\R^ 3))} ).
\end{align*}
We again employ \eqref{Delta-equals-secondgrad}  and the Ladyzhenskaya inequality \eqref{Lady} and obtain
\begin{equation}\label{estimateMt}
\|\partial_t M_m\|_{L^{2}(0,t;L^2(\Omega;\R^3))} \leq C(T, \text{IED}, H_\mathrm{ext}),
\end{equation}
which implies that
\begin{equation}\label{estimatenablaMt}
\|\partial_t \nabla M_m\|_{L^{2}(0,t;W^{-1,2}(\Omega;\R^{3\times 2}))} \leq C(T, \text{IED}, H_\mathrm{ext}).
\end{equation}
Let us make one more observation on $\partial_t \nabla M_m$. To this end set
$$
W^{1,2}_n(\Omega;\R^{3\times 2})) := \{X \in W^{1,2}(\Omega;\R^{3\times 2})): Xn=0 \text{ a.e. on }\partial \Omega\};
$$
here, recall that $n$ denotes the outer normal to the boundary of $\Omega$. Notice that $\nabla M_m \in W^{1,2}_n(\Omega;\R^{3\times 2}))$, hence, we would like to deduce that $\partial_t \nabla M_m \in L^2(0,t; \big(W^{1,2}_n(\Omega;\R^{3\times 2})\big)')$ for $W^{1,2}_n(\Omega;\R^{3\times 2})) \hookrightarrow L^2((\Omega;\R^{3\times 2})) \hookrightarrow \big(W^{1,2}_n(\Omega;\R^{3\times 2})\big)'$ to form a Gelfand triple (cf. e.g. \cite{Roubicek2013}). This is indeed true, since for any $g \in L^2(\Omega;\R^2)$ we have that $\nabla g \in \big(W^{1,2}_n(\Omega;\R^{3 \times 2})\big)'$. To see this, let us take $g$ smooth at first and any arbitrary $\Phi$ in $W^{1,2}_n(\Omega;\R^{3\times 2})$. Then $\Phi n=0$ on $\partial \Omega$ and we obtain
$$
\int_\Omega \nabla g \cdot \Phi \, dx = -\int_\Omega g (\nabla \cdot \Phi) \, dx \leq \|g\|_{L^2(\Omega;\R^2)} \|\Phi\|_{W^{1,2}_n(\Omega;\R^{3 \times 2})},
$$
so that the claim follows by approximation. This calculation also shows that
\begin{equation}
\label{estimatenablaMtDual}
\|\partial_t \nabla M_m\|_{L^{2}(0,t;(W_n^{-1,2}(\Omega;\R^{3\times 2}))'} \leq C(T, \text{IED}, H_\mathrm{ext}).
\end{equation}

Finally, we consider $\partial_t F_m$. To this end, let us take any arbitrary 
$$
\|\xi\|_{L^4(0,t)} \leq 1 \qquad \text{and} \qquad \|\zeta\|_{W^{1,2}(\Omega; \R^{2 \times 2})} \leq 1
$$
and estimate
\begin{align*}
& 
\int_0^{t} \llangle \partial_t F_m , \zeta  \rrangle \xi~ds 
\leq  
\int_0^{t}\!\!\!\!\! \int_\Omega \left| (v_m\cdot\nabla) F_m \cdot (\xi\zeta) \right| + \left| (\nabla v_m F_m) \cdot (\xi\zeta) \right| + \kappa \left| \nabla F_m \cdot (\xi\nabla\zeta)\right|~dx~ds \\
& \qquad \leq
\int_0^{t} \|v_m \|_{L^3(\Omega;\R^2)} \|\nabla F_m\|_{L^2(\Omega;\R^{2 \times 2 \times 2)})} |\xi| \|\zeta\|_{L^6(\Omega;\R^{2 \times 2})} + \|\nabla v_m \|_{L^2(\Omega;\R^{2 \times 2} )} \|F_m\|_{L^3(\Omega;\R^{2 \times 2})} |\xi| \|\zeta\|_{L^6(\Omega;\R^{2 \times 2})} \\
&~~~~~~~~~~~~~~~~~~~~~ + \kappa \|\nabla F_m \|_{L^2(\Omega; \R^{2 \times 2 \times 2})} |\xi| \|\nabla\zeta\|_{L^2(\Omega; \R^{2 \times 2 \times 2})}~ds \\
& \qquad 
\leq  \|v_m \|_{L^4(0,t;L^3(\Omega; \R^2))} \|\nabla F_m\|_{L^2(0,t;L^2(\Omega; \R^{2 \times 2 \times 2}))} \\
&~~~~~~~~~~~~~~~~~~~~~ + \|\nabla v_m \|_{L^2(0,t;L^2(\Omega; \R^{2 \times 2}))} \|F_m\|_{L^4(0,t;L^3(\Omega;\R^{2 \times 2} ))} + \kappa \|\nabla F_m \|_{L^{\frac43}(0,t;L^2(\Omega; \R^{2 \times 2 \times 2}))},
\end{align*}
where again all terms are bounded by the energetic estimates \eqref{a-prioriI} when taking also interpolation inequalities, analogous to those that we used in the balance of momentum, into account. In total, we obtain that
\begin{equation}\label{estimateFt}
\|\partial_t F_m\|_{L^{\frac43}(0,t;W^{-1,2}(\Omega;\R^{2\times2}))} \leq C(\text{IED}).
\end{equation}
Notice that the dual estimate \eqref{estimateFt} that we obtained for $F_m$ is slightly worse than those that we got for $M_m$ and $v_m$ in \eqref{estimatenablaMtDual} and \eqref{a-prioriTimeDerV}. Hence, proving that $F_m$ attains the right initial data will be slightly more difficult, see Step~6.

\noindent{\bf Step 4: Extending the approximate solution}\\
The approximate solution and the a-priori estimates that we obtained so far only hold on a short interval $[0,t^\ast)$. Nevertheless, they can be extended to the interval $[0,T)$ with $T$ as in Theorem~\ref{ThmSln}. Indeed, we may find a time instant $t_\ast$ such that $t_\ast$ is arbitrarily close to $t^\ast$ and $(v_m(t_\ast), F_m(t_\ast), M_m(t_\ast))$ are well defined and bounded in $L^2(\Omega; \R^ 2)  \times L^2(\Omega; \R^{2 \times 2})\times W^{1,2}(\Omega;\R^3)$ by $\mathrm{IED}$, cf.\ \eqref{a-prioriI}--\eqref{a-prioriI-final}.
Moreover, due to \eqref{FinalEst-LinftyH2M}, we can assure that $M_m(t_\ast)$ is bounded in the $W^{2,2}$-norm by a constant that only depends on $m$, $L$ (which in fact is only dependent on $\mathrm{IED}$) and $\mathrm{IED}$. Notice that since $m$ is fixed for the moment, this gives a uniform bound on the $W^{2,2}$-norm of the magnetization with respect to $m$, which is needed in Step 1.

Thus, we may regard $(v_m(t_\ast), F_m(t_\ast), M_m(t_\ast))$ as new initial data and repeat the procedure from Step 1. In Step 1, we saw that the solution interval depends only on $m$, the norms of the initial data and global properties of the external magnetic field, and thus on $\mathrm{IED}$ in our case. Hence, we conclude that there exists a constant $\delta>0$ which depends only on $m$, IED and the external field such that the system \eqref{LLGPARTweakformFapproxm}--\eqref{LLGPARTbasrepvapproxmODE} has a solution $(v_m, F_m, M_m)$ on $\Omega\times[t_\ast,t_\ast+\delta)$ coinciding with the earlier solution $(v_m, F_m, M_m)(t_\ast)$ in $t_\ast$.

Gluing the two solutions together, we thus obtain a solution on a time interval $[0,t_\ast+\delta)$. Repeating the procedure in Steps 2 and 3 then gives that the same a-priori estimates hold for the prolonged solution on the solution interval  $[0,t_\ast+\delta)$. Notice that by repeating this procedure on the whole interval  $[0,t_\ast+\delta)$ and not just on $[t_\ast,t_\ast+\delta)$ allows us to bound $(v_m(t), F_m(t), M_m(t))$ in $L^2(\Omega; \R^ 2)  \times L^2(\Omega; \R^{2 \times 2})\times W^{1,2}(\Omega;\R^3)$ for almost all $t \in [0,t_\ast+\delta)$ by IED, i.e., by the initial data and the external field, and not just by the norms of $(v_m, F_m, M_m)(t_\ast)$ and the external field.  

Thus, we can continue the extension on another time instant of length $\delta$ which is the same as above. This is due to the fact that the initial data for this extension will again be bounded by $\mathrm{IED}$.  Finally, we obtain a solution $(v_m, F_m, M_m)$ on $\Omega\times(0,T)$.

\noindent{\bf Step 5: Convergence of the approximate system}\\
From the a-priori estimates \eqref{a-prioriI} and \eqref{a-prioriII} obtained in Step 2 as well as the dual a-priori estimates \eqref{a-prioriTimeDerV}, \eqref{estimatenablaMtDual} and \eqref{estimateFt} from Step 3, we conclude by the Aubin-Lions Lemma (cf., e.g., \cite{Roubicek2013}) that, up to a non-relabeled subsequence, there exist $(v,F,M) \in L^2(0,T; W^{1,2}_{0,\mathrm{div}}(\Omega;\R^2)) \times L^2(0,T;W^{1,2}(\Omega;\R^{2 \times 2})) \times L^2(0,T;W^{2,2}(\Omega; \R^3))$ such that
\begin{eqnarray}
& v_m\to v &\quad\text{in } L^2(0,T;L^4(\Omega;\R^2)), \label{convvm} \\
& \nabla v_m \rightharpoonup \nabla v &\quad\text{in } L^2(0,T;L^2(\Omega;\R^{2\times2})), \label{convnablavm}\\
& F_m\to F &\quad\text{in } L^2(0,T;L^4(\Omega;\R^{2\times2})), \label{convFm} \\
& \nabla F_m \rightharpoonup \nabla F &\quad\text{in } L^2(0,T;L^2(\Omega;\R^{2\times2\times2})) \label{convnablaFm} \\
& \nabla M_m\to \nabla M &\quad\text{in } L^2(0,T;L^4(\Omega;\R^{3\times2})), \label{convnablaMm} \\
& \Delta M_m\rightharpoonup \Delta M &\quad\text{in } L^2(0,T;L^2(\Omega;\R^3)), \label{convDeltaMm} .
\end{eqnarray}
Moreover, due to the continuous embedding of $W^{1,4}(\Omega; \R^{3}) \hookrightarrow L^\infty(\Omega; \R^{3})$, we also have that
\begin{equation}
M_m \to  M \quad\text{in } L^2(0,T;L^\infty(\Omega;\R^3)). \label{convLinftyM}
\end{equation}
At this point, we are ready to pass to the limit in the equations \eqref{v-EqDisc}, \eqref{LLGPARTweakformFapproxm} and \eqref{LLGPARTweakformMapproxm} that form together the discrete system. Let us start with the balance of momentum \eqref{v-EqDisc}. To this end, let us choose some arbitrary $\zeta\in W^{1,2}_{0,\mathrm{div}}(\Omega;\R^2)$ and use $\zeta_m:=P_m(\zeta)\in\H_m$ as a test function in \eqref{v-EqDisc}. Moreover,  multiply this equation by $\xi\in W^{1,\infty}(0,T)$ with $\xi(T)=0$ and integrate over $[0,T)$ to obtain
\begin{align}
\nonumber &  \int_0^{T} \! \!\!\!\!\int_\Omega -v_m{\cdot} \zeta_m \xi' + (v_m{\cdot}\nabla)v_m \cdot (\xi\zeta_m) + \nu \nabla v_m \cdot (\xi\nabla\zeta_m) - \left(\nabla M_m\odot\nabla M_m - W'(F_m)F_m^\top\right) {\cdot} (\xi\nabla\zeta_m), \\
&~~~~~~~~~~~~~~~~~~~~~~~~~~~~~~ - ((\nabla H_\mathrm{ext})^\top M_m )\cdot(\xi\zeta_m)~dx~dt =\int_\Omega v_m(0)\cdot (\xi(0)\zeta_m)~dx, \label{eqconvvm}
\end{align}
where we used integration by parts with respect to time.

Applying the continuity of the Nemytski\u{i} mapping induced by $W'(\cdot)$ (cf., e.g., \cite{Roubicek2013}), we get that $W'(F_m) \to W'(F)$ in $L^2(0,T;L^4(\Omega;\R^{2\times2}))$. Therefore, by standard weak-strong convergence arguments we get that \eqref{eqconvvm} converges to 
\begin{align}
\nonumber &  
\int_0^{T} \! \!\!\!\!\int_\Omega -v{\cdot} \zeta \xi' + (v{\cdot}\nabla)v\cdot (\xi\zeta) + \nu \nabla v \cdot (\xi\nabla\zeta) - \left(\nabla M\odot\nabla M - W'(F)F^\top\right) {\cdot} (\xi\nabla\zeta)~dx~dt \\
&~~~~~~~~~~~~~~~~~~~~~~~~~~~~~~ - ((\nabla H_\mathrm{ext})^\top M )\cdot(\xi\zeta)~dx~dt =\int_\Omega v_0\cdot (\xi(0)\zeta)~dx \label{eqconvv}
\end{align}
as $m\to\infty$. 

Further, multiply \eqref{LLGPARTweakformFapproxm} by $\xi \in W^{1,\infty}(0,T)$ with $\xi(T)=0$ and integrate over $[0,T)$ to get
\begin{align}
  \int_0^{T}\! \!\!\!\!\int_\Omega - F_m \cdot (\xi'\Xi) + ((v_m\cdot\nabla)F_m -\nabla v_m F_m)\cdot (\xi\Xi)  +  \kappa \nabla F_m \cdot (\xi\nabla\Xi)~dx~dt = \int_\Omega F_0 \cdot (\xi(0)\Xi)~dx, \label{eqconvFm}
\end{align}
where we used that $F_m(0)=F_0$ and that, due to Lemma \ref{LLGPARTLemParabolicEq}, $\partial_t F_m \in L^2(0,T;W^{-1,2}(\Omega; \R^{2 \times 2}))$ and simultaneously $F_m \in L^2(0,T;W^{1,2}_0(\Omega; \R^{2 \times 2}))$ whence $F_m \in C(0,T; L^2(\Omega;\R^{2 \times 2})$ and the by-parts integration formula holds; cf., e.g., \cite[Lemma~7.3]{Roubicek2013}. Then, after integrating by parts in time, the duality pairing between $W^{-1,2}(\Omega; \R^{2 \times 2})$ and $W^{1,2}_0(\Omega; \R^{2\times 2})$ actually reduces to a scalar product on $L^2(\Omega;\R^{2 \times 2})$. 

Standard weak-strong convergence arguments allow us to identify the limit as 
\begin{align}
  \int_0^{T}\! \!\!\!\!\int_\Omega - F \cdot (\xi'\Xi) + ((v\cdot\nabla)F - \nabla v F) \cdot (\xi\Xi)  +  \kappa \nabla F \cdot (\xi\nabla\Xi)~dx~dt = \int_\Omega F_0 \cdot (\xi(0)\Xi)~dx
\end{align}
as $m\to\infty$. 

Finally, we pass to the limit in the LLG. By multiplying \eqref{LLGPARTweakformMapproxm} by $\tilde{\zeta} \in L^2(\Omega; \R^3)$ and $\xi \in W^{1,\infty}(0,T)$ with $\xi(T) = 0$ and integrating over space and time, we obtain with $M_m(0) = M_0$ that
\begin{align}
\nonumber &  \int_0^T\!\!\!\!\!\int_\Omega - M_m\cdot (\xi'\tilde\zeta) + \big((v_m\cdot\nabla) M_m + (M_m \times (\Delta M_m+H_\mathrm{ext}) - \Delta M_m\big)\cdot (\xi\tilde\zeta) ~dx~dt   \\
&~~~~~~~~~~~~~~~~~~ = \int_0^{T}\!\!\!\!\!\int_\Omega \big(|\nabla M_m |^2 M_m - M_m(M_m{\cdot}H_\mathrm{ext}) + H_\mathrm{ext} \big) \cdot (\xi\tilde\zeta)\, dxdt + \int_\Omega M_0\cdot (\xi(0)\tilde\zeta)~dx. \label{eqconvMm}
\end{align}
As $m\to\infty$, this equation converges to 
\begin{align}
\nonumber &  \int_0^T\!\!\!\!\!\int_\Omega - M\cdot (\xi'\tilde\zeta) + \big((v_m\cdot\nabla) M + (M \times (\Delta M+H_\mathrm{ext}) - \Delta M\big)\cdot (\xi\tilde\zeta) ~dx~dt   \\
&~~~~~~~~~~~~~~~~~~ = \int_0^{T}\!\!\!\!\!\int_\Omega \big(|\nabla M |^2 M - M(M{\cdot}H_\mathrm{ext}) + H_\mathrm{ext} \big) \cdot (\xi\tilde\zeta)\, dxdt + \int_\Omega M_0\cdot (\xi(0)\tilde\zeta)~dx.
\end{align}

Indeed, for the term $\int_0^{T}\!\!\int_\Omega |\nabla M_m |^2 M_m \cdot (\xi\tilde\zeta) \, dx dt$ this is obtained by the following calculation
\begin{align*}
& \left| \int_0^T\!\!\!\!\!\int_\Omega \big(| \nabla M_m |^2 M_m - |\nabla M |^2 M \big) \cdot (\xi\tilde\zeta) \, dx dt \right| \\
&\qquad  = \left| \int_0^T\!\!\!\!\!\int_\Omega \big((|\nabla M_m |^2 - |\nabla M |^2)M_m + | \nabla M |^2 (M_m - M )\big) \cdot (\xi\tilde\zeta)~dx~dt \right| \\
&\qquad  = \left| \int_0^T\!\!\!\!\!\int_\Omega \big((\nabla M_m - \nabla M){\cdot}(\nabla M_m + \nabla M)M_m + | \nabla M |^2 (M_m - M )\big) \cdot (\xi\tilde\zeta)~dx~dt \right| \\
&\qquad\leq  \| \nabla M_m + \nabla M\|_{L^2(0,T;L^4(\Omega;\R^{3 \times 2}))} \|\nabla M_m - \nabla M\|_{L^2(0,T;L^4(\Omega;\R^{3 \times 2}))} \\
&~~~~~~~~~~~~~~~~~~~~~ \times \|M_m\|_{L^\infty(0,T;L^\infty(\Omega;\R^3))} \|\xi\|_{L^\infty(0,T)}\|\tilde\zeta\|_{L^4(\Omega;\R^3))} \\
&~~~~~~~~~~~~~~~ + \|\nabla M\|^2_{L^2(0,T;L^4(\Omega;\R^{3\times2}))} \| M_m - M\|_{L^2(0,T;L^\infty(\Omega))} \|\xi\|_{L^\infty(0,T)}\|\tilde\zeta\|_{L^2(\Omega))},
\end{align*}
where the second term on the right hand side tends to zero owing to \eqref{convLinftyM} while the first term on the right hand side vanishes thanks to \eqref{convnablaMm}.

All other terms converge by a combination of weak and strong convergences in \eqref{convvm}--\eqref{convLinftyM}. Hence, the discrete solution that we constructed in Step~1 and extended in Step~4 converges in the sense of \eqref{convvm}--\eqref{convLinftyM} to a solution of \eqref{weakformv}--\eqref{weakformM}.

Moreover, the $L^\infty$-in-time regularities in Definition~\ref{def-weakSol} hold by the lower semicontinuity of norms, and since the estimate \eqref{a-prioriI-final} is uniformly in $m$ and is obtained for the entire time interval $(0,T)$.

\noindent{\bf Step 6: Attainment of the initial data}

Finally, we are left to prove that the initial data is actually attained by the solution in the sense of Definition~\ref{def-weakSol}. As for $v$ and $M$ this is fairly easy because the a-priori estimates \eqref{a-prioriI}, \eqref{a-prioriII}, \eqref{a-prioriTimeDerV}, \eqref{estimatenablaMtDual} and \eqref{estimateFt}  translate by weak lower semicontinuity to the limit so that by
\begin{align*}
v \in L^2(0,T;W^{1,2}_0(\Omega;\R^2)) \quad  &\text{and} \quad \partial_t v \in L^2(0,T;W^{-1,2}(\Omega;\R^2)),\\
M \in L^2(0,T;L^2(\Omega;\R^3)) \quad  &\text{and} \quad \partial_t M \in L^2(0,T;L^2(\Omega;\R^3)),\\
\nabla M \in L^2(0,T;W^{1,2}_n(\Omega;\R^{3 \times 2})) \quad  &\text{and} \quad \partial_t \nabla M \in L^2(0,T;(W^{1,2}_n(\Omega;\R^{3 \times 2}))'),
\end{align*}
and by, e.g., \cite[Lemma~7.3]{Roubicek2013} we have that $v\in C(0,T;L^2(\Omega;\R^2))$ and $M \in C(0,T; W^{1,2}(\Omega;\R^3))$. Moreover, we can see directly from \eqref{weakformv} that $v(0)=v_0$ a.e. in $\Omega$. Indeed, for some $\varepsilon >0$ take $\phi(x,t)=\phi_1(t)\phi_2(x)$ in such a way that $\phi_1(0)=1$, $\phi_1(t)$ linear on $(0,\varepsilon)$ and $\phi_1(t)=0$ for all $t \in [\varepsilon,T]$ while $\phi
_2\in W^{1,2}_{0,\mathrm{div}}(\Omega;\R^2)$ is arbitrary. Then, as $\varepsilon \to 0$ we have $\phi(\cdot,t) \to 0$ a.e.\ in $\Omega$ while $\partial_t \phi(t) \rightharpoonup - \delta_0$ in measures, where $\delta_0$ denotes the Dirac measure centered at $0$. Thus,
$$
\int_\Omega (v(0)-v_0) \cdot \phi_2 dx = 0,
$$
for all $\phi_2 \in W^{1,2}_{0,\mathrm{div}}(\Omega;\R^2)$, which shows the claim. The situation is analogous for $M$.

For $F$, the situation is slightly more complicated since the obtained integrability of the time derivative does not allow us to immediately form a Gelfand triple since $L^{4/3}$ (in-time integrability of the time derivative of $F$) is not dual to $L^2$ (in-time integrability of $\nabla F$). Nevertheless, we conclude from the a-priori estimates \eqref{a-prioriI}  and \eqref{estimateFt} that (notice that we actually get from \eqref{a-prioriI} that $F \in L^\infty(0,T; L^2(\Omega;\R^{2\times2})$ yields the first statement below)
\begin{align*}
F  \in L^{4}(0,T;W^{-1,2}(\Omega;\R^{2\times2})) \quad  \text{and} \quad \partial_t F  \in L^{\frac43}(0,T;W^{-1,2}(\Omega;\R^{2\times2})),
\end{align*}
which implies that (see e.g. \cite[Lemma~7.3]{Roubicek2013}) $F \in C(0,T;W^{-1,2}(\Omega;\R^{2\times2}))$; combining this with the fact that $F \in L^\infty(0,T;L^2(\Omega;\R^{2 \times 2}))$,
we have (see e.g. \cite[Chapter~III, Lemma~1.4]{Temam1977}) that
\begin{equation*}
F \in C(0,T;L_w^2(\Omega;\R^{2 \times 2})),
\end{equation*}
where $L_w^2(\Omega;\R^{2 \times 2})$ is the space of $L^2$-functions whose values are $2\times 2$-matrices and which are equipped with the weak topology. Moreover, by the same procedure as above, we may identify that $F(0) = F_0$ whence 
\begin{align} \label{Ftweak}
F(t) \rightharpoonup F_0 \quad \text{in } L^2(\Omega;\R^{2 \times 2}) \quad \text{as }t\to0_+.
\end{align}
By the convexity of $W$ this translates to
$$
\int_\Omega W(F_0) dx \leq \liminf_{t\to 0_+} \int_\Omega W(F(t)) dx.
$$
On the other hand, the energy estimate \eqref{energy-Est} also translates to the limit by weak$^*$ lower semicontinuity of the energy with respect to the convergence of $(v_m,F_m,M_m) \in L^\infty(0,T;L^2(\Omega;\R^2)) \times L^\infty(0,T;L^2(\Omega;\R^{2\times2}))\times L^\infty(0,T;W^{1,2}(\Omega;\R^3))$. Hence
\begin{align*}
&\int_\Omega \frac{1}{2} |v(t)|^2 + \frac{1}{2}|\nabla M(t)|^2 - M(t) \cdot H_\mathrm{ext}(t) + W(F) \, dx + \int_0^t\!\!\!\int_\Omega \kappa a |\nabla F|^2 + \nu |\nabla v|^ 2 \, dx\,dt 
\\ &\qquad \leq \int_\Omega \frac{1}{2} |v_0|^2 + \frac{1}{2}|\nabla M_0|^2 - M_0 \cdot H_\mathrm{ext}(0) + W(F_0) \, dx - \int_0^t\!\!\!\int_\Omega M\cdot  \partial_t H_\mathrm{ext} \, dx\,dt,
\end{align*}
for almost all $t \in(0,T)$. By continuity, we may extend the estimate to hold even for all $t \in (0,T)$. Thus, taking the $\limsup_{t \to 0^+}$ and using the already proved attainment of initial data (as well as the continuity of the external field in time) we get that
$$
\limsup_{t \to 0^+} \int_\Omega W(F(t)) dx \leq \int_\Omega W(F_0) dx,
$$
so that altogether $\int_\Omega W(F(t)) dx  \to \int_\Omega W(F_0) dx$. By the convexity and growth of $W$ this means that 
$$
\lim_{t \to 0^+} \|F(t)\|_{L^2(\Omega;\R^{2 \times 2})} = \|F_0\|_{L^2(\Omega;\R^{2 \times 2})},
$$
which combined with the already obtained weak convergence of $F(t)$ to $F_0$ in \eqref{Ftweak} means that the initial data are attained in the strong sense as claimed.
\end{proof}

\section{Proofs of Lemma \ref{LLGPARTLemParabolicEq} and Lemma \ref{lemma-cont}}
\label{sec:PfsLem}

\begin{proof}[Proof of Lemma \ref{LLGPARTLemParabolicEq}] 
Recall that for a fixed $v \in V_m(t_0)$, we aim to construct $(F,M)$ satisfying
\begin{align}
& \llangle {\partial_t} F , \Xi \rrangle + \int_\Omega (v \cdot\nabla)F : \Xi - (\nabla v F) : \Xi + \kappa \nabla F \tp \nabla\Xi~dx = 0 & &\text{ in } (0,t_1),  \label{Feq-Iterations1}\\
 & \nonumber \partial_t M + (v\cdot\nabla) M = |\nabla M|^2M + \Delta M \\& \qquad \qquad \qquad \qquad - M \times (\Delta M + H_\mathrm{ext}) - M(M\cdot H_\mathrm{ext}) + H_\mathrm{ext}  & &\text{ in } \Omega\times(0,t_1), 
\label{Meq-Iterations1}
\end{align}
for all  $\Xi\in W^{1,2}_0(\Omega;\R^{2\times 2})$, together with the initial conditions \eqref{initialF}--\eqref{initialM} and boundary conditions \eqref{simpleboundaryF}--\eqref{boundaryMsimple}.

Notice that the two equations \eqref{Feq-Iterations1} and \eqref{Meq-Iterations1} are decoupled. Consequently, we can prove existence separately. To prove the existence, we rely on similar methods as in the proof of Theorem \ref{ThmSln}; i.e., we use a Galerkin approximation and standard ODE theory to prove existence of approximate solutions. Thus, existence of solutions is proved at first on some short time interval $[0,\tilde t)$ for some $0<\tilde t\leq t_1$, but we can extend the solution later to the entire interval $[0,t_1]$ due to the a priori estimates obtained.

\noindent{\bf Existence of weak solution to \eqref{Feq-Iterations1}:} As for the Galerkin approximation, we project $F$ and the equation \eqref{Feq-Iterations1} on finite dimensional subspaces of the eigenfunctions of the Laplace-operator that form an orthonormal basis of $L^2(\Omega; \R^{2 \times 2})$ and an orthogonal basis of $W^{1,2}(\Omega; \R^{2 \times 2})$.
Let $\overline{P_k}: L^2(\Omega;\R^{2\times 2})\to  \{ \Xi_1,\Xi_2,\ldots,\Xi_k \}$ be this orthonormal projection. 

For a fixed $k\in\N$, we look for a function $F_k$ of the form
\begin{equation}\label{basrepFapproxm}
F_k(x,t) = \sum_{i=1}^k d_k^i(t)\Xi_i(x)
\end{equation}
solving the projection of \eqref{Feq-Iterations1} on the $\mathrm{span} \{ \Xi_1,\Xi_2,\ldots,\Xi_k \}$; i.e. the ODE
\begin{equation}\label{basrepFapproxmODE}
\frac{d}{dt} d_k^i(t) = -\kappa\mu_i d_k^i(t) + \sum_{j=1}^k d_k^j(t) \tilde A_{j}^i(t), \qquad i=1,\ldots,k,
\end{equation}
where
\begin{equation}
\label{basrepFapproxmOpA}
\tilde A_{j}^i(t) = -\int_\Omega (v(x,t)\cdot\nabla)\Xi_j(x) : \Xi_i(x) - (\nabla v(x,t) \Xi_j(x)) : \Xi_i(x)~dx.
\end{equation}
The initial condition becomes
\begin{equation}\label{initialphiiapproxmODE}
d_k^i(0) = \int_\Omega F_0(x):\Xi_i(x)~dx
\end{equation}
for $i=1,\ldots,k$.
We apply Carath\'eodory's existence theorem (see, e.g., \cite[Chapter~1, Theorem~1]{Filippov1988}) to obtain  absolutely continuous solution $d_n^i(t)$ of \eqref{basrepFapproxmODE} on the interval $[0,\tilde{t})$. Notice that the solution interval will thus depend only on the intial condition and the $L^\infty(0,t_1; W^{1,\infty}(\Omega;\R^2))$-norm of $v$; i.e. $m$ and $L$. Notice also that, since the right-hand side of \eqref{basrepFapproxmODE} is locally Lipschitz, the obtained solution is unique.

We now prove all the needed a-priori estimates. To this end, let us first sum \eqref{basrepFapproxmODE} over all $i=1\ldots k$ to get 
\begin{equation}\label{Fpdelemproof}
\int_\Omega \big(\partial_t F_k  + (v\cdot\nabla)F_k - \nabla v F_k) \big) \cdot \Xi dx + \int_\Omega \kappa \nabla F_k \cdot \nabla\Xi dx = 0,
\end{equation}
for all $\Xi \in \mathrm{span} \{ \Xi_1,\Xi_2,\ldots,\Xi_n \}$. Let us now test \eqref{Fpdelemproof} by $F_k$ and integrate over $[0,t]$ for $t\leq \tilde t$ to find
\begin{align*}
\frac12\int_\Omega |F_k (t)|^2~dx + \int_0^t \!\!\!\int_\Omega (v\cdot\nabla)\frac{|F_k|^2}{2} - \nabla v \cdot (F_kF_k^\top) + \kappa |\nabla F_k|^2~dx~ds = \frac12\int_\Omega | \overline{P}_k(F_0) |^2 dx \leq \frac12\int_\Omega | F_0 |^2 dx
\end{align*}
As the second expression vanishes because $v$ is divergence free, we get, by rearranging,
\begin{align}\label{FestimateONE}
\nonumber  &\frac12 \int_\Omega |F_k(t)|^2~dx + \int_0^t\!\!\!\int_\Omega \kappa |\nabla F_k|^2~dx~ds 
\leq  \int_0^t\!\!\!\int_\Omega | \nabla v : (F_kF_k^\top)|~dx~ds + \frac12\int_\Omega | F_0 |^2~dx \\
&\qquad \leq C(L,m)+\|F_0\|_{L^2(\Omega;\R^{2\times2})} + \frac{\kappa}{2}\int_0^t\!\!\!\int_\Omega |F_k|^2~dx~ds,
\end{align}
where in the last line we used that $ \| \nabla v \|_{L^\infty(0,t;L^\infty(\Omega;\R^2))} \leq C(L,m)$ since $v \in V_m(t_0)$.
Applying Gronwall's inequality yields that
\begin{equation}
\|F_k\|_{L^\infty(0,\tilde t;L^2(\Omega;\R^{2\times2})) \cap L^2(0,\tilde t;W^{1,2}(\Omega;\R^{2 \times 2})) } \leq C(L,m)+\|F_0\|_{L^2(\Omega;\R^{2\times2})}.
\end{equation}
Notice that from this estimate it follows that we may extend the approximate solution onto the interval $[0,t_0)$ by the same procedure as in Step 4 of the proof of Theorem \ref{ThmSln}.
Next, we derive an estimate on the time derivative $\partial_t F_k$ in $L^2(0,t_0; W^{-1,2}(\Omega;\R^{2\times 2}))$. To this end, let us choose some arbitrary $\zeta \in L^2(0,t_0)$ and $\Xi \in W^{1,2}_0(\Omega;\R^{2 \times 2})$ satisfying 
$$
\|\zeta\|_{L^2(0,t_0)}\leq 1 \qquad \text{and} \qquad \|\Xi\|_{W^{1,2}_0(\Omega;\R^{2\times 2})} \leq 1
$$
and calculate
\begin{align*}
&\int_0^{t_0} \partial_t F_k \cdot (\overline{P_k}\Xi)\zeta~dt = \int_0^{t_0} \partial_t F_k \cdot \Xi\zeta~dt =
\int_0^{t_0} \!\!\!\! \int_\Omega -(v\cdot\nabla) F_k : (\zeta\Xi) + (\nabla v F_k) : (\zeta\Xi) - \kappa \nabla F_k \tp (\zeta\nabla\Xi)~dx~dt \\
&\qquad \leq
\int_0^{t_0} \Big( \big(\|v\|_{L^\infty(\Omega;\R^2)} \|\nabla F_k\|_{L^2(\Omega;\R^{2 \times 2 \times 2})} + \|\nabla v\|_{L^\infty(\Omega; \R^{2 \times 2})} \|F_k\|_{L^2(\Omega;\R^ {2 \times 2})}\big)\|\Xi\|_{L^2(\Omega;\R^{2 \times 2})} \\ &\qquad \qquad + \kappa \|\nabla F_k \|_{L^2(\Omega;\R^{2 \times 2 \times 2})} \|\nabla\Xi\|_{L^2(\Omega;\R^{2 \times 2 \times 2 \times 2})} \big) |\zeta| ~dt \\
&\qquad \leq C(L,m) \|F_k\|_{L^2(0,t_0; W^{1,2}(\Omega;\R^{2 \times 2})};
\end{align*}
and since for $\|F_k\|_{L^2(0,t_0; W^{1,2}(\Omega;\R^{2 \times 2})}$ we already got an estimate in \eqref{FestimateONE}, we see that 
\begin{equation}\label{dualestimateF}
\|\partial_t F_k \|_{L^2(0,t_0; W^{-1,2}(\Omega;\R^{2\times 2}))} \leq C(L,m) + \|F_0\|_{L^2(\Omega;\R^{2\times2})}.
\end{equation}
 From the preceding estimates, we see that may extract a subsequence (not relabeled) from $(F_k)_{n \in \mathbb{N}}$ such that
%
%
\begin{eqnarray}
& F_k \rightharpoonup F &\quad\text{in } L^2(0,t_0;L^2(\Omega;\R^{2\times 2})), \label{convphiin} \\
& \partial_t F_k\rightharpoonup \partial_t F &\quad\text{in } L^2(0,t_0;W^{-1,2}(\Omega;\R^{2\times 2})), \label{convddtphiin} \\
& \nabla F_k \rightharpoonup \nabla F &\quad\text{in } L^2(0,t_0;L^2(\Omega;\R^{2\times 2\times 2})).\label{convnablaphiin}
\end{eqnarray}
As, by fixing $v$, \eqref{Fpdelemproof} is a linear, we may pass with $k \to \infty$ to get that $F$ solves \eqref{Feq-Iterations1}. Owing to the linearity once again, this is the unique solution of \eqref{Feq-Iterations1}.

\noindent{\bf Existence of weak solutions to \eqref{Meq-Iterations1}}: Let us now prove existence of solutions as well as suitable a-priori estimates for \eqref{Meq-Iterations1}. The procedure to obtain those is inspired by \cite{CarbouFabrie2001}. As above, we perform a Galerkin approximation; to this end, let $\{\eta_i\}_{i=1}^\infty \subset C^\infty(\overline{\Omega};\R^3)$ be an orthonormal basis of $L^2(\Omega;\R^3)$ and an orthogonal basis of $W^{2,2}_n(\Omega;\R^3)$,
where 
$$
W^{2,2}_n(\Omega;\R^3) = \{X\in W^{2,2}_n(\Omega;\R^3): \nabla X n=0 \text{ a.e. on }\partial \Omega \}.
$$
For example, this basis may be composed of eigenfunctions of the operator $\Delta^2 + \mathrm{id}$ subject to vanishing Neumann boundary condition for the eigenfunction and its Laplacian. Let $\widetilde P_k:L^2(\Omega;\R^3)\to \mathrm{span} \{ \eta_1,\eta_2,\ldots,\eta_k \}$ be the orthonormal projection onto finite dimensional subspaces formed by this basis. For a fixed $k\in\N$, we look for a function $M_k$ of the form
\begin{equation}\label{LLGPARTbasrepMapproxm}
M_k(x,t) = \sum_{i=1}^k h_k^i(t)\eta_i(x).
\end{equation}
that satisfies the projection of \eqref{Meq-Iterations1} onto $\mathrm{span} \{ \eta_1,\eta_2,\ldots,\eta_k \}$; this amounts to solving the following ODE 
\begin{align}\label{LLGPARTbasrepMapproxmODE}
\frac{d}{dt} h_k^i(t) & = \sum_{j=1}^k h_k^j(t) \hat A_{j}^i(t) + \sum_{j,l=1}^k h_k^j(t) h_k^l(t) \hat B_{jl}^i + \sum_{j,l,m=1}^k h_k^j(t) h_k^l(t) h_k^m(t) \hat C_{jlm}^i,\qquad i=1,\ldots,k,
\end{align}
where
\begin{align}
\label{LLGPARTbasrepMapproxmOpA}
\hat A_{j}^i(t) & = -\int_\Omega \big( (v(x,t)\cdot\nabla)\eta_j(x) + (\eta_j(x)\times \Hext(x,t) ) - \Delta\eta_j(x) - \Hext(x,t) \big) \cdot \eta_i(x)
~dx, \\
\label{LLGPARTbasrepMapproxmOpB}
\hat B_{jk}^i & = -\int_\Omega (\eta_j(x) \times \Delta\eta_k(x) + (\eta_k(x)\cdot\Hext) \eta_j(x) ) \cdot \eta_i(x)~dx, \\
\label{LLGPARTbasrepMapproxmOpC}
\hat C_{jkl}^i & = \int_\Omega (\nabla\eta_j(x) : \nabla\eta_k(x)) (\eta_l(x) \cdot \eta_i(x))~dx.
\end{align}
The initial condition becomes
\begin{equation}\label{LLGPARTinitialMapproxmODE}
h_k^i(0) = \int_\Omega M_0(x) \cdot \eta_i(x)~dx, \qquad i=1,\ldots,n.
\end{equation}
Existence of unique Lipschitz continuous solutions $h_n^i(t)$ is also here obatined by Carath\'eodory's existence theorem on a time interval $[0,t^{**})$. Notice that the length of the solution interval depends just on the $L^2(\Omega;\R^3)$ of $\Hext$ (which is controled by assumption uniformly on $[0,T]$) and the $L^\infty(\Omega; \R^2)$ norm of $v$ (which is controled uniformly by $C(m,L)$ on $[0,t_0]$ since $v \in V_m(t_0)$).

In order to deduce suitable a-priori estimates, we first rewrite \eqref{LLGPARTbasrepMapproxmODE} as 
\begin{equation}
\int_\Omega \big(\partial_t M_k + (v\cdot\nabla) M_k - |\nabla M_k|^2M_k - \Delta M_k + M_k \times (\Delta M_k + H_\mathrm{ext}) + M_k(M_k\cdot H_\mathrm{ext}) - H_\mathrm{ext} \big) \eta \, dx = 0
\label{Mapprox-eq}
\end{equation}
for all $\eta \in \mathrm{span} \{ \eta_1,\eta_2,\ldots,\eta_k \}$. Let us first test \eqref{Mapprox-eq} by $M_k$ to obtain 
\begin{align}\label{LLGPARTMLemmaEstimate1}
\nonumber  \frac{d}{dt} &\frac{1}{2} \int_\Omega |M_k|^2 \, dx + \int_\Omega |\nabla M_k| \, dx  =  \int_\Omega |\nabla M_k|^2 |M_k|^2 - |M_k|^2(M_k\cdot\Hext) + M_k\cdot\Hext~dx \\
&\leq 2\left( \|M_k\|_{L^\infty(\Omega;\R^3)}^2 \|\nabla M_k\|_{L^2(\Omega;\R^{3\times2})}^2 + \big(\|M_k\|_{L^\infty(\Omega;\R^3)}^2 +1\big)\|M_k\|_{L^\infty(\Omega;\R^3)}\|\Hext\|_{L^1(\Omega;\R^3)} \right).
\end{align}
Next, we test \eqref{Mapprox-eq} by $\Delta^2 M_n$ and obtain for all $t \in [0,t^{**})$
\begin{align}\label{LLGPARTMLemmaEstimate2}
\nonumber& \frac12\frac{d}{dt} \int_\Omega |\Delta M_k|^2 \, dx + \int_\Omega |\nabla\Delta M_k|^2 \, dx \\  &\qquad \leq 
\underbrace{\int_\Omega |(v\cdot\nabla)M_k \cdot \Delta^2 M_k| \, dx}_{:= \mathrm{I}_1} + \underbrace{\int_\Omega |(M_k\times (\Delta M_n + \Hext)) \cdot \Delta^2 M_k| \, dx}_{:= \mathrm{I}_2} + \underbrace{\int_\Omega\big||\nabla M_k|^2 M_k \cdot \Delta^2 M_k \big| \, dx}_{:= \mathrm{I}_3} \nonumber \\ &\qquad \qquad + \underbrace{ \int_\Omega |(M_k\cdot\Hext)M_k \cdot \Delta^2 M_k| \, dx}_{:= \mathrm{I}_4} + \underbrace{\int_\Omega|\Hext\cdot\Delta^2 M_k|~dx}_{:= \mathrm{I}_5}
\end{align}
We will estimate the integrals $\mathrm{I}_1$--$\mathrm{I}_5$ separately. To do so, we will utilize the following estimates which hold for all smooth $M:\Omega \to \R^3$ ($\Omega \subset \R^2$) with zero Neumann boundary conditions:
\begin{align}
\label{LLGPARTcarbouest21}
\|M\|_{W^{2,2}(\Omega;\R^3)} & \leq C \left( \|M\|_{L^2(\Omega;\R^3)}^2 + \|\Delta M\|_{L^2(\Omega;\R^3)}^2 \right)^{\frac12}\\
 \|\nabla M\|_{L^4(\Omega;\R^{3\times 2})} & \leq C \|\nabla M\|_{L^2(\Omega;\R^{3\times 2})}^{\frac12} \left( \|\nabla M\|_{L^2(\Omega;\R^{3\times 2})}^2 + \|\Delta M\|_{L^2(\Omega;\R^3)}^2 \right)^{\frac14}
\label{LLGPARTcarbouest561}\\
 \|\nabla M\|_{L^6(\Omega;\R^{3\times 2})} & \leq C \|\nabla M\|_{L^2(\Omega;\R^{3\times 2})}^{\frac13}\left( \|\nabla M\|_{L^2(\Omega;\R^{3\times 2})}^2 + \|\Delta M\|_{L^2(\Omega;\R^3)}^2 \right)^{\frac13}
\label{LLGPARTcarbouest562}\\
\|\nabla M\|_{L^\infty(\Omega;\R^{3\times 2})} & \leq C \|\nabla M\|_{L^2(\Omega;\R^{3\times 2})}^{\frac12}\left( \|\nabla M\|_{L^2(\Omega;\R^{3\times 2})}^2 + \|\Delta M\|_{L^2(\Omega;\R^3)}^2 + \|\nabla\Delta M\|_{L^2(\Omega;\R^{3\times 2})}^2 \right)^{\frac14}
\label{LLGPARTcarbouest563}
\\
\|\Delta M\|_{L^4(\Omega;\R^3)} & \leq C \|\Delta M\|_{L^2(\Omega;\R^3)}^{\frac12} \left( \|\Delta M\|_{L^2(\Omega;\R^3)}^2 + \|\nabla\Delta M\|_{L^2(\Omega;\R^{3\times 2})}^2 \right)^{\frac14} \label{LLGPARTcarbouest564}
\end{align}
for some constant $C > 0$ depending just on $\Omega$. Indeed, \eqref{LLGPARTcarbouest21} is a variant of the Poincar\'{e} inequality after realizing that $\|\nabla^2 M\|_{L^2(\Omega;\R^{3 \times 2 \times 2})}^2 = \|\Delta M\|_{L^2(\Omega;\R^3)}^2$ by integration by parts due to the vanishing Neumann boundary conditions. Further, \eqref{LLGPARTcarbouest562} and \eqref{LLGPARTcarbouest564} are variants of the Ladyzhenskaya inequality formulated here for functions the traces of which do not necessarily vanish on $\partial \Omega$ while \eqref{LLGPARTcarbouest562} is a more general interpolation inequality obtained from the Gagliardo-Nirenberg theorem. Finally, \eqref{LLGPARTcarbouest563} is a variant of the Agmon inequality valid in 2D.  

We start to estimate the term $\mathrm{I}_1$ and get, since $v\in V_m(t_0)$,
\begin{align}
\nonumber \mathrm{I}_1 &\leq \int_\Omega |(\nabla v(\nabla M_k)^\top) \cdot \nabla\Delta M_k| + |(v\cdot\nabla)\nabla M_k \cdot \nabla\Delta M_k|~dx \\
& \leq  \|\nabla v\|_{L^\infty(\Omega;\R^{2 \times 2})} \|\nabla M_k\|_{L^2(\Omega; \R^{3 \times 2})} \|\nabla\Delta M_k\|_{L^2(\Omega; \R^{3 \times 2})}  + \|v\|_{L^\infty(\Omega; \R^2)} \|\nabla^2 M_k\|_{L^2(\Omega; \R^{3 \times 2 \times 2})} \|\nabla\Delta M_k\|_{L^2(\Omega; \R^{3 \times 2})} \nonumber\\
 &\leq C(L,m)\left( \|\nabla M_k\|_{L^2(\Omega; \R^3)} + \|\Delta M_k\|_{L^2(\Omega; \R^3)} \right) \|\nabla\Delta M_k\|_{L^2(\Omega; \R^{3 \times 2})}, \label{LLGPARTestI1}
\end{align}
where we used \eqref{Delta-equals-secondgrad} and \eqref{LLGPARTcarbouest21}.
For the integral term $\mathrm{I}_2$, we obtain
\begin{align}
\nonumber \mathrm{I}_2 &\leq \int_\Omega |(\nabla M_k\times (\Delta M_k + \Hext)) \cdot \nabla\Delta M_k| + |(\nabla M_k\times (\Hext+\nabla\Hext)) \cdot \nabla\Delta M_k|~dx \\ \nonumber
&\leq\|\nabla M_k\|_{L^4(\Omega;\R^{3\times 2})} \big(\|\Delta M_k\|_{L^4(\Omega;\R^3)} +2\|\Hext\|_{W^{1,4}(\Omega;\R^3)}\big) \|\nabla\Delta M_k\|_{L^2(\Omega;\R^{3\times 2})} \\
\nonumber &\leq C \|\nabla M_k\|_{L^2(\Omega;\R^{3\times 2})}^{\frac12}\left( \|\nabla M_k\|_{L^2(\Omega;\R^{3\times 2})}^2 + \|\Delta M_k\|_{L^2(\Omega;\R^3)}^2 \right)^{\frac14} \nonumber \\
 &~~~\times  \bigg(\|\Delta M_k\|_{L^2(\Omega;\R^3)}^{\frac12} \left( \|\Delta M_k\|_{L^2(\Omega;\R^3)}^2 + \|\nabla\Delta M_k\|_{L^2(\Omega;\R^{3\times 2})}^2 \right)^{\frac14} + 2\|\Hext\|_{W^{1,4}(\Omega;\R^3)} \bigg) \|\nabla\Delta M_k\|_{L^2(\Omega;\R^{3\times 2})}.  \label{LLGPARTestI2}
\end{align}
We estimate the integral term $\mathrm{I}_3$ and find out that
\begin{align}
\nonumber \mathrm{I}_3 &= \int_\Omega |(2  (\nabla^2 M_k \nabla M_k )  \otimes M_k) \cdot \nabla\Delta M_k| + |\nabla M_k|^2 | \nabla M_k \cdot \nabla\Delta M_k|~dx \\ 
&\leq 2 \|M_k\|_{L^\infty(\Omega;\R^3)} \|\nabla M_k\|_{L^\infty(\Omega;\R^{3\times 2})} \|\nabla^2 M_k\|_{L^2(\Omega;\R^{3\times 2\times 2})} \|\nabla\Delta M_k\|_{L^2(\Omega;\R^{3\times 2})} \nonumber \\&\qquad + \|\nabla M_k\|_{L^6(\Omega;\R^{3\times 2})}^3 \|\nabla\Delta M_k\|_{L^2(\Omega;\R^{3\times 2})} \nonumber\\
\nonumber &\leq C \Big(\|M_k\|_{L^\infty(\Omega;\R^3)} \|\Delta M_k\|_{L^2(\Omega;\R^3)} \|\nabla M_k\|_{L^2(\Omega;\R^{3\times 2})}^{\frac12} \left( \|\nabla M_k\|_{L^2(\Omega;\R^{3\times 2})}^2 + \|\Delta M_k\|_{L^2(\Omega;\R^3)}^2 + \|\nabla\Delta M_k\|_{L^2(\Omega;\R^{3\times 2})}^2 \right)^{\frac14}
\\&\qquad +\|\nabla M_k\|_{L^2(\Omega;\R^{3\times 2})} \left( \|\nabla M_k\|_{L^2(\Omega;\R^{3\times 2})}^2 + \|\Delta M_k\|_{L^2(\Omega;\R^3)}^2 \right)  \Big) \|\nabla\Delta M_k\|_{L^2(\Omega;\R^{3\times 2})}.\label{LLGPARTestI3}
\end{align}
For the integral term $\mathrm{I}_4$, we estimate
\begin{align}
\nonumber \mathrm{I}_4 & = \int_\Omega |(M_k\cdot\Hext)(\nabla M_k \cdot \nabla\Delta M_k)| + |(\nabla\Delta M_k)^ \top M_k)\cdot ((\nabla M_k)^\top\Hext + (\nabla\Hext)^\top M_k)|~dx \\
& \leq  \|M_k\|_{L^\infty(\Omega; \R^3)} \|\Hext\|_{W^{1,3}(\Omega;\R^3)} \big(2\|\nabla M_k\|_{L^6(\Omega;\R^ {3 \times 2})}+\|M_k\|_{L^\infty(\Omega; \R^3)}\big) \|\nabla\Delta M_k\|_{L^2(\Omega;\R^ {3 \times 2})} \nonumber \\
&\leq \|M_k\|_{L^\infty(\Omega; \R^3)} \|\Hext\|_{W^{1,3}(\Omega;\R^3)} \big(2\|\nabla M_n\|_{L^2(\Omega)}^2+2\|\Delta M_k\|_{L^2(\Omega; \R^3)}^2+\|M_k\|_{L^\infty(\Omega; \R^3)}\big) \|\nabla\Delta M_k\|_{L^2(\Omega;\R^ {3 \times 2})}, \label{LLGPARTestI4}
\end{align}
where to get the last expression we used \eqref{LLGPARTcarbouest562} combined with the Young inequality. Finally, estimating the integral term $I_5$ yields
\begin{align}
\mathrm{I}_5
&\leq \|\nabla\Hext\|_{L^2(\Omega)} \|\nabla\Delta M_n\|_{L^2(\Omega)}. \label{LLGPARTestI5}
\end{align}
Combining \eqref{LLGPARTestI1}--\eqref{LLGPARTestI5}, we obtain from \eqref{LLGPARTMLemmaEstimate2} and an iterative application of Young's inequality that
\begin{align}\label{LLGPARTMLemmaEstimateH2Extensionpreconv}
\nonumber& \frac12\frac{d}{dt} \int_\Omega |\Delta M_k|^2 \, dx + \frac12\int_\Omega |\nabla\Delta M_k|^2 \, dx \leq  C(L,m) \left( \|\nabla M_k\|_{L^2(\Omega;\R^{3\times 2})}^4 + \|\Delta M_k\|_{L^2(\Omega;\R^3)}^4 \right) \\&\qquad + C(\Hext)\big(1+\|M_k\|_{L^\infty(\Omega;\R^3)}^4\big)\big(1+\|\nabla M_k\|_{L^2(\Omega;\R^{3 \times 2})}^2\big) \left( \|\nabla M_k\|_{L^2(\Omega;\R^{3\times 2})}^2 + \|\Delta M_k\|_{L^2(\Omega;\R^3)}^2 +1 \right)^ 2 
\end{align}
We shall make use of \eqref{LLGPARTMLemmaEstimateH2Extensionpreconv} later to derive \eqref{LinftyH2onM}. But in order to derive further a-priori estimates, let us use that $W^{2,2}(\Omega; \R^3)$ embeds continuously into $W^{1,2}(\Omega; \R^3)$ as well as $L^\infty(\Omega;\R^3)$ so that with the help of \eqref{LLGPARTcarbouest21} we have that
\begin{align*}
\|\nabla M_k\|_{L^2(\Omega;\R^{3\times 2})}&\leq C \left( \|M\|_{L^2(\Omega;\R^3)}^2 + \|\Delta M\|_{L^2(\Omega;\R^3)}^2 \right)^{\frac12}, \\
\|M_k\|_{L^\infty(\Omega;\R^3)}&\leq C \left( \|M\|_{L^2(\Omega;\R^3)}^2 + \|\Delta M\|_{L^2(\Omega;\R^3)}^2 \right)^{\frac12}.
\end{align*}
Using this in \eqref{LLGPARTMLemmaEstimateH2Extensionpreconv} and adding it to \eqref{LLGPARTMLemmaEstimate1} lets us deduce that
\begin{align}\label{LLGPARTMLemmaEstimatetotal}
\nonumber & \frac{d}{dt} \int_\Omega |M_k|^2 + |\Delta M_k|^2 dx + \int_\Omega |\nabla M_n|^2 + |\nabla\Delta M_n|^2 dx 
\\& \qquad \leq  C(L,m,\Hext) \left( 1 + \left( \|M_k\|_{L^2(\Omega;\R^3)}^2 + \|\Delta M_n\|_{L^2(\Omega;\R^3)}^2 \right)^{8} \right).
\end{align}
In the next step, we make use of the following classical comparison lemma (see \cite[Lemma~2.4]{CarbouFabrie2001}), which we state without a proof:
\begin{lemma}\label{LLGPARTcomparisonlemma}
Let $f:\R\times\R\to\R$ be $C^1$ and nondecreasing in its second variable. Assume further that $y:I\subset\R\to\R$ is a continuous function satisfying $y(t)\leq y_0 + \int_0^t f(s,y(s))~ds$ for all $t>0$. Let $z:I\to\R$ be the solution of $z'(t)=f(t,z(t))$, $z(0)=y_0$. Then, it holds $y(t)\leq z(t)$ for all $t>0$.
\end{lemma}
From \eqref{LLGPARTMLemmaEstimatetotal} and Lemma~\ref{LLGPARTcomparisonlemma} we deduce the existence of a time $0<t_1\leq t^{\ast\ast}$ such that
\begin{align}\label{LLGPARTMboundtotal}
& \|M_k\|_{L^\infty(0,t_1;W^{1,2}(\Omega;\R^3))} + \|M_k\|_{L^2(0,t_1;W^{3,2}(\Omega;\Hext))} \leq C(L,m,M_0,\Hext).
\end{align}

In order to be able to pass to the limit as $k \to \infty$ in  \eqref{Mapprox-eq}, we need to derive further estimates on the time derivative of $M_k$ as well as of $\nabla M_k$. To this end, let us test \eqref{Mapprox-eq} by $\partial_t M_k$ to get
\begin{align*}
&\int_\Omega |\partial_t M_k|^2~dx
=  \int_\Omega  \big(- (v\cdot\nabla)M_k - (M_k\times\Delta M_k) + |\nabla M_k|^2 M_k  + \Delta M_k - (M_k\cdot\Hext) M_k + \Hext\big)\cdot\partial_t M_k~dx \\
 &\qquad \leq 3 \int_\Omega |(v\cdot\nabla)M_k |^2 + |M_k\times\Delta M_k|^2 + |\nabla M_k|^4 |M_k|^2 + |\Delta M_k|^2   + (M_k\cdot\Hext)^2 |M_k|^2 + |\Hext|^2~dx \\&\qquad \qquad + \frac{1}{2} \int_\Omega |\partial_t M_k|^2~dx.
\end{align*}
From there, we get
\begin{align*}
\|\partial_t M_k\|_{L^2(\Omega;\R^3)}
\leq & 6 \Big( C(L,m) \|\nabla M_k\|_{L^2(\Omega;\R^3)}^2 + \|M_k\|_{L^\infty(\Omega;\R^3)} \|\Delta M_k\|_{L^2(\Omega;\R^3)}^2 + \|\nabla M_k\|_{L^4(\Omega;\R^{3 \times 2})}^4 \|M_k\|_{L^\infty(\Omega;\R^3)}^2 \\
&~~~ + \|\Delta M_k\|_{L^2(\Omega;\R^3)}^2 + \|\Hext\|_{L^2(\Omega;\R^3)}^2 \|M_k\|_{L^\infty(\Omega; \R^3)}^4 + \|\Hext\|_{L^2(\Omega; \R^3)}^2 \Big),
\end{align*}
where we take the supremum over all $t\in[0,t_1)$ to find, using \eqref{LLGPARTMboundtotal} and the fact that $\Hext\in C(0,T;L^2(\Omega; \R^3))$,
\begin{equation}\label{LLGPARTMtbound1}
\|\partial_t M_k\|_{L^\infty(0,t_1;L^2(\Omega; \R^3))} \leq C(L,m,M_0,\Hext).
\end{equation}
Next, we test \eqref{Mapprox-eq} by $ -\partial_t \Delta  M_k$ and integrate over $(0,t)$ for $t\leq t_1$ to find out that
\begin{align*}
& \int_0^t \| \partial_t \nabla M_k\|_{L^2(\Omega; \R^{3
 \times 2})}^2~ds + \frac12 \big(\|\Delta M_k (t)\|_{L^2(\Omega; \R^3)}^2- \|\Delta M_k (0)\|_{L^2(\Omega; \R^
 )}^2\big)\\
& \quad = \int_0^t \!\!\!\int_\Omega \big( (v\cdot\nabla)M_k + (M_k\times\Delta M_k) - |\nabla M_k|^2 M_k - (M_k\cdot\Hext) M_k + \Hext\big)\cdot\partial_t \Delta  M_k~dx~ds \\
& \quad = \int_0^t \!\!\!\int_\Omega  - \big((\nabla v (\nabla M_k)^\top)  - (v\cdot\nabla)\nabla M_k  - \nabla ( M_k\times\Delta M_k) + (2 (\nabla^2 M_k \nabla M_k ) \otimes   M_k  + |\nabla M_k|^2 \nabla M_k \\&\qquad \quad + (M_k\cdot\Hext)\nabla M_k - \nabla\Hext \big) \cdot \partial_t \nabla M_k   + ((\partial_t \nabla M_k)^\top M_k)\cdot ((\nabla M_k)^\top\Hext + (\nabla \Hext)^\top M_k) ~dx~ds \\
 & \qquad \leq \frac{1}{2} \int_0^t \|\partial_t \nabla M_k\|_{L^2(\Omega;\R^{3 \times 2})}^2  + 5 \int_0^t \!\!\!\int_\Omega C(L,m) \big(|\nabla M_k|^2 + |\nabla^2 M_k|^2\big) +|\nabla M_k|^2|\Delta M_k|^2 + |M_k|^2|\nabla \Delta M_k|^2
&~~~~~~~~~ + |(\nabla M_k\times\Delta M_k)|^2 + |M_k\times\nabla\Delta M_k|^2 \\
&~~~~~~~~~ + 4 |M_k|^2 |\nabla M_k|^2 |\nabla^2 M_k|^2 + |\nabla M_k|^6    + 2|M_k|^2 |\nabla M_k|^2|\Hext|^2  + |\nabla\Hext|^2 + |M_k|^4|\nabla \Hext|^2~dx~ds \\
 & \qquad \leq \int_0^t \|\partial_t \nabla M_k\|_{L^2(\Omega;\R^{3 \times 2})}^2 + 5 \Big(C(L,m) \big(\|\nabla M_k\|_{L^2(0,t_1;L^2(\Omega;\R^{3 \times 2}))}^2 + \|\nabla^2 M_k\|_{L^2(0,t_1;L^2(\Omega; \R^{3 \times 2 \times 2}))}^2\big) \\
&~~~~~~~~~ + (1+4\|M_k\|_{L^\infty(0,t_1 t;L^\infty(\Omega, \R^3))}^2)\|\nabla M_k\|_{L^\infty(0,t_1;L^4(\Omega; \R^{3 \times 2}))} \|\nabla^2 M_k\|_{L^2(0,t_1;L^4(\Omega;\R^{3 \times 2 \times 2}))} \\
&~~~~~~~~~ + \|M_k\|_{L^\infty(0,t_1;L^\infty(\Omega;\R^3))}^2 \|\nabla\Delta M_k\|_{L^2(0,t_1;L^2(\Omega;\R^{3 \times 2}))}^2 
 + \|\nabla M_k\|_{L^6(0,t_1;L^6(\Omega; \R^{3 \times 2}))}^6 + \|\nabla\Hext\|_{L^2(0,t_1;L^2(\Omega; \R^{3 \times 2}))}^2 \\
&~~~~~~~~~ +  2\|M_k\|_{L^\infty(0,t_1;L^\infty(\Omega; \R^3))}^2 2\|\Hext\|_{L^\infty(0,t_1;L^2(\Omega; \R^3))}^2 \|\nabla M_k\|_{L^2(0,t_1;L^\infty(\Omega; \R^{3 \times 2}))}^2 \\
&~~~~~~~~~ + \|M_k\|_{L^\infty(0,t_1;L^\infty(\Omega; \R^3))}^4\|\nabla\Hext\|_{L^2(0,t_1;L^2(\Omega; \R^{3 \times 2}))}^2 \Big).
\end{align*}
Taking the supremum over all $t\in[0,t_1)$ and using \eqref{LLGPARTMboundtotal}, we get the bound
\begin{equation}\label{LLGPARTMtbound2}
\|\partial_t \nabla M_k\|_{L^2(0,t_1;L^2(\Omega))} \leq C(v,M_0,\Hext).
\end{equation}

We now pass to the limit as $k\to\infty$ to obtain a weak solution to equation \eqref{Meq-Iterations1}. By our a-priori estimates , we can find $M \in L^\infty(0,t_1,W^{2,2}(\Omega; \R^3)\cap W^{1,\infty}(0,t_1; L^2(\Omega; \R^3) \cap L^2(0,t_1; W^{3,2}(\Omega; \R^3)$ such that for a (non-relabeled) subsequence of $(M_k)_{k \in \N}$, we have that
\begin{eqnarray}
& M_k \to M &\quad\text{in } L^p(0,t_1; W^{1,2}(\Omega;\R^{3})), \quad 1<p<\infty, \label{LLGPARTconvMnSTRONG} \\
& \partial_t M_k \rightharpoonup   \partial_t M &\quad\text{in } L^2(0,t_1;W^{1,2}(\Omega;\R^{3})). \label{LLGPARTconvddtMn}
\end{eqnarray}
Indeed, the weak convergence result follow by the Banach-Alaoglu theorem; while the strong convergence \eqref{LLGPARTconvMnSTRONG} is obtained from the Aubin-Lions lemma. In fact, the Aubin-Lions lemma yields at first the strong convergence $M_k \to M$ in $L^2(0,t_1;W^{2,2}(\Omega;\R^3))$  but combining this with the boundedness of $(M_k)_{k \in \N}$ $L^\infty(0,t_1;W^{2,2}(\Omega;\R^3))$ gives \eqref{LLGPARTconvMnSTRONG}.

Thus, multiplying \eqref{Mapprox-eq} with $\zeta \in L^2(0,t_1)$, integrating over $(0,t_1)$, and passing to the limit $k \to \infty$ yields the equation
\begin{align*}
& \int_0^{t_1}\!\!\!\int_\Omega \big(\partial_t M  + (v\cdot\nabla)M + (M \times \Delta M) - |\nabla M |^2 M - \Delta M + (M\cdot\Hext)M + \Hext\big) \cdot \varphi\zeta~dx~dt
\end{align*}
which holds for all $\varphi \in L^2(\Omega; \R^3)$ and all $\zeta \in L^2(0,t_1)$. From this, we can conclude that $M$ satisfies \eqref{Meq-Iterations1}. 

Furthermore, notice that $M\in L^\infty(0,t_1,W^{2,2}(\Omega; \R^3))\cap W^{1,\infty}(0,t_1; L^2(\Omega; \R^3)) \cap L^2(0,t_1; W^{3,2}(\Omega; \R^3))$ is the unique solution of \eqref{Meq-Iterations1}. Indeed,  assume that there existed two solutions $M_1, M_2 \in L^\infty(0,t_1,W^{2,2}(\Omega; \R^3))\cap W^{1,\infty}(0,t_1; L^2(\Omega; \R^3)) \cap L^2(0,t_1; W^{3,2}(\Omega; \R^3))$ with $M_1 \neq M_2$. The difference $M_1-M_2$ would then fulfill for almost all $x \in \Omega$ and almost all $t \in [0,t_1)$
\begin{align*}
&\partial_t M_1 - M_2 + (v\cdot\nabla)(M_1 - M_2)
= \Delta(M_1 - M_2) -(M_1 - M_2)\times\Delta M_1 + M_2\times(\Delta(M_1 - M_2))  \\ &\quad + \left(|\nabla M_1|^2 - |\nabla M_2|^2\right)M_1 + |\nabla M_2|^2(M_1 - M_2)  - (M_1 - M_2)(M_1\cdot\Hext) - M_2 ((M_1 - M_2)\cdot\Hext).
\end{align*}
We multiply this equation by $(M_1-M_2)$, integrate over $\Omega$ and use the identity $(a\times b)\cdot c = (b\times c)\cdot a$ to find out that
\begin{align*}
& \frac12 \frac{d}{dt} \int_\Omega|M_1 - M_2|^2 dx  + \int_\Omega |\nabla(M_1 - M_2)|^2 dx \\
 &\quad = \int_\Omega \big( (\Delta(M_1 - M_2))\times (M_1 - M_2) \big) \cdot M_2  + \left( \left(\nabla M_1 - \nabla M_2\right) \cdot \left(\nabla M_1 + \nabla M_2\right) \right)M_1 \cdot (M_1 - M_2) \\
&\quad \quad  + |\nabla M_2|^2 |M_1 - M_2|^2 - |M_1 - M_2|^2 (M_1\cdot\Hext) - M_2\cdot(M_1 - M_2)((M_1 - M_2)\cdot\Hext)~dx\\
 &\quad = \frac{1}{2}\int_\Omega |\nabla(M_1 - M_2)|^2 dx + \underbrace{\big(|M_2|^2+|\nabla (M_1+M_2)|^2|M_2|^2 + |\nabla M_2|^2 +(|M_1|+|M_2|)|\Hext| \big)}_{(\star)}|M_1 - M_2|^2 dx,
\end{align*}
where we integrated by parts in the first term on the second line. Now, due to the assumed regularity of $M_1$ and $M_2$, we know that $(\star)$ is bounded in $L^1(0,t_1; L^\infty(\Omega))$ which allow us to apply the Gronwall lemma. Thus, since $M_1(0) = M_2(0)$, the two solutions coincide.

Moreover, let us show that $M\in L^\infty(0,t_1,W^{2,2}(\Omega; \R^3))\cap W^{1,\infty}(0,t_1; L^2(\Omega; \R^3)) \cap L^2(0,t_1; W^{3,2}(\Omega; \R^3))$, the solution \eqref{Meq-Iterations1}, fulfills that $|M(t)|=1$ a.e. in $\Omega$ for a.a. $t\in [0,t_1)$.
To this end, let us multiply \eqref{Meq-Iterations1} with $M$ to obtain
\begin{align}
 \frac{1}{2}\Big(\partial_t |M|^2 + (v\cdot\nabla) |M|^2 - \Delta |M|^2 \Big)= (|M|^2-1)( |\nabla M|^2-M\cdot H_\mathrm{ext}) \qquad \text{a.e. in $\Omega\times [0,t_1)$} \label{EqMSq}
\end{align} 
Notice that \eqref{EqMSq} is solved by $|M|=1$ so we just need to show that this is the unique solution. Let us set $\theta:=|M|^2$ and since $M$ is fixed being the unique solution of \eqref{Meq-Iterations1}, we may denote $f(M):=|\nabla M|^2-M\cdot H_\mathrm{ext}$. Thus \eqref{EqMSq} transfers to an equation for $\theta$ that reads
\begin{equation}
 \frac{1}{2}\Big(\partial_t \theta + (v\cdot\nabla) \theta - \Delta \theta \Big)= (\theta-1)f(M) \qquad \text{a.e. in $\Omega\times [0,t_1)$ with $\theta(0)=|M_0|^2=1$;} 
 \label{EqMSq2}
\end{equation}
now if \eqref{EqMSq2} had two solutions $\theta_1, \theta_2 \in L^\infty(0,t_1,W^{2,2}(\Omega))\cap W^{1,\infty}(0,t_1; L^2(\Omega))$ we could subtract \eqref{EqMSq} for $\theta_1$ and $\theta_2$, multiply by $\theta_1-\theta_2$ and conclude by the Gronwall lemma that the two solutions have to coincide. 

Finally, we pass to the limit in the inequality in \eqref{LLGPARTMLemmaEstimateH2Extensionpreconv} integrated over $(0,t_1)$. On the left-hand side we rely on the convexity of the norm, while on the right-hand side it is enough to use the strong convergence \eqref{LLGPARTconvMnSTRONG}. Therefore, since $|M|=1$ a.e. in $\Omega \times [0,t_1)$ in the limit, we obtain for almost all $t\in [0,t_1)$
\begin{align*}
& \nonumber \|\Delta M(t)\|_{L^2(\Omega;\R^3)}^2  \leq \|\Delta M_0\|_{L^2(\Omega;\R^3)}^2 \\&\qquad + C(L,m,\Hext) \int_0^{t}1+\|\nabla M_k\|_{L^2(\Omega;\R^{3\times 2})}^6 + \|\nabla M_k\|_{L^2(\Omega;\R^{3\times 2})}^2\|\Delta M_k\|_{L^2(\Omega;\R^3)}^4 ~ds.
\end{align*}
\end{proof}

\begin{proof}[Proof of Lemma \ref{lemma-cont}]
We show that $\mathcal{L}$ defined in \eqref{operatorL} is continuous on $V_m(t^\ast)$ in the topology of $C(0,t^\ast;\H_m)$. To this end, let $(v_l)_{l \in \N} \subset V_m(t^\ast)$ converge to some $v\in V_m(t^\ast)$ in the sense that $(g_m^i)_l \to g_m^i$ in $C(0,t^\ast)$ for $i=1,\ldots,m$, where $v_l = \sum_{i=1}^m (g^i_m)_l(t)\xi_i(x)$ and $v = \sum_{i=1}^m g^i_m(t)\xi_i(x)$.

Let us denote by $(F_l, M_l)$ and $(F,M)$ the solutions of \eqref{Feq-Iterations}-\eqref{Meq-Iterations} corresponding to $v_l$ and $v$, respectively.  Notice that their  existence is guaranteed by Lemma~\ref{LLGPARTLemParabolicEq}. 

Let us first realize that $F_l \to F$ in  $L^\infty(0,t^\ast;L^2(\Omega))$. To this end, subtract \eqref{Feq-Iterations} for $F$ from \eqref{Feq-Iterations} for $F_l$, test the result by $F_l-F$ and integrate of over $(0,t)$ with some $0\leq t\leq t^\ast $to obtain 
\begin{align*}
&\frac12\int_\Omega |F_l - F|^2(t)~dx + \int_0^t\!\!\!\int_\Omega \kappa |\nabla (F_l - F)|^2~dx~ds = - \frac12 \int_0^t\!\!\!\int_\Omega (v_l\cdot\nabla)|F_l - F|^2  dx ds  \\ &\qquad + \int_0^t\!\!\!\int_\Omega (\nabla v_l (F_l - F)) \cdot (F_l - F) - ((v_l-v)\cdot\nabla)F \cdot (F_l - F) + (\nabla v_l - \nabla v)F \cdot (F_l - F)~dx~ds,
\end{align*}
where we used that $F_l$ and $F$ have the same initial data. Realizing that $\int_0^t\!\!\!\int_\Omega \kappa |\nabla (F_l - F)|^2~dx~ds = 0$ because $v_l$ is divergence free and employing the Young's inequality yields that
\begin{align}\label{estimateNablaDeltaFDIFF}
 \nonumber \left\| (F_l - F)(t) \right\|_{L^2(\Omega;\R^{2 \times 2})}^2
& \leq C\int_0^t \|((v_l-v)\cdot\nabla)F\|_{L^2(\Omega;\R^{2 \times 2})}^2 + \|(\nabla v_l - \nabla v)F\|_{L^2(\Omega;\R^{2 \times 2})}^2 ds \\ &\qquad \quad + \int_0^t \|(F_l - F)(s)\|_{L^2(\Omega;\R^{2 \times 2})}^2~ds,
\end{align}
where the first integral on the right hand side vanishes as $l\to\infty$ due to the assumed convergence of $(v_l)_{l\in \N}$. The claim thus follows by the Gronwall inequality.

Next, we check that $M_l \to M$ in $L^2(0,t^\ast;W^{1,2}(\Omega;\R^3))$. Similarly as above, we subtract \eqref{Meq-Iterations} for $M$ from \eqref{Meq-Iterations} for $M_l$ to have that for a.a. $x \in \Omega$ and a.e. $t \in [0,t^\ast)$:
\begin{align}\label{LLGPARTContOpMeq}
\nonumber& \partial_t (M_l - M) - \Delta(M_l - M) + (v_l\cdot\nabla)(M_l - M) + ((v_l - v)\cdot\nabla)M =  -(M_l - M)\times(\Delta M_l + \Hext) \\ &\quad + M\times(\Delta(M_l - M)) + \left(|\nabla M_l|^2 - |\nabla M|^2\right)M_l + \big(|\nabla M|^2-(M_l{\cdot}\Hext)\big)(M_l - M) - M ((M_l - M){\cdot}\Hext);
\end{align}
further multiply the result by $M_l-M$ and integrate over $\Omega$ and $(0,t)$  with some $0\leq t\leq t^\ast$ to get
\begin{align*}
& \frac12 \int_\Omega |M_l - M|^2 (t)~dx + \int_0^t \!\!\! \int_\Omega |\nabla(M_l-M)|^2~dx~ds \\
 &\quad \leq \int_0^t \!\!\! \int_\Omega \big(|v_l - v||\nabla M|+| \nabla (M_l-M)|(2|\nabla M| + |\nabla M_l|)\big) |(M_l - M)|  + \big( |\nabla M|^2 + 2|\Hext|\big) |M_l - M|^2 ~dx~ds
 \end{align*}
 Using now the Young's inequality, we obtain
\begin{align}\label{LLGPARTestimateMDIFF}
\nonumber& \frac12 \int_\Omega |M_l - M|^2 (t)~dx + \int_0^t \!\!\! \int_\Omega |\nabla(M_l-M)|^2~dx~ds \leq \int_0^t\!\!\!\int_\Omega |v_l - v||\nabla M| dx ds \\
 & + C \int_0^t \Big( 1 + \|\nabla M\|_{L^\infty(\Omega;\R^{3\times 2})}^2 +  \|\nabla M_l\|_{L^\infty(\Omega;\R^{3\times 2})}^2 + \|\Hext\|_{L^\infty(\Omega;\R^3)} \Big) \|M_l - M\|_{L^2(\Omega;\R^3)}^2~ds,
\end{align}
from which the claim follows by the Gronwall inequality.

Let us define $(D_m^i)_l(t)$ and $D_m^i(t)$ via \eqref{LLGPARTbasrepvapproxmOp} by using $(F_l,M_l)$ and $(F,M)$, respectively. Notice that due to the already proved convergence of $(F_l)_{l \in \N}$ to $F$ and $(M_l)_{l\in \N}$ to $M$, we see that $(D_m^i)_l \to D_m^i$ in $L^{1}(0,t^\ast)$. 

Further take $\tilde{v}_l = \sum_{i=1}^m (\tilde{g}^i_m)_l(t)\xi_i(x)$ as $\mathcal{L}(v_l)$ and $\tilde{v} = \sum_{i=1}^m (\tilde{g}^i_m)(t)\xi_i(x)$ as $\mathcal{L}(v)$, i.e., the solutions of \eqref{LLGPARTbasrepMapproxmODE} with $(D_m^i)_l$ and $D_m^i$  as the right hand side, respectively. The proof is finished if we can show that $(\tilde{g}_m)_l \to \tilde{g}_m$ uniformly on $[0,t^\ast)$. To this end, subtract \eqref{LLGPARTbasrepMapproxmODE} for $\tilde{v}$ from the one for $\tilde{v}_l$ and write in matrix notation
\begin{align*}
& \partial_t \big( (\tilde g_m)_l(t) - \tilde g_m(t) \big) = -\nu\operatorname{diag}(\lambda_1,\ldots,\lambda_m) \big( (\tilde g_m)_l(t) - \tilde g_m(t) \big) + (D_m)_l(t) - D_m(t)\\ \qquad \quad  & + \big(A^1(\tilde g_m)_l(t)\cdot (\tilde g_m)_l(t),\ldots, A^m(\tilde g_m)_l(t)\cdot (\tilde g_m)_l(t)\big)  - \big(A^1\tilde g_m(t)\cdot \tilde g_m(t),\ldots,A^m\tilde g_m(t)\cdot \tilde g_m(t)\big) 
\end{align*}
Adding and subtracting the vector $\big(A^1\tilde g_m(t)\cdot (\tilde g_m)_l(t),\ldots,A^m\tilde g_m(t)\cdot \tilde g_m(t)\big) $ and integrating over $(0,t)$ with some $0 \leq t \leq t^ \ast$ gives
\begin{align*}
 \left| (\tilde g_m)_l(t) - \tilde g_m(t) \right| \leq 
 C(L,m) \int_0^t \left| (\tilde g_m)_l(s) - \tilde g_m(s) \right|~ds + \int_0^t \left|(D_m(s))_l - D_m(s)\right|~ds,
\end{align*}
and, by means of the Gronwall inequality, we obtain,
\begin{equation*}
\left| (\tilde g_m)_l(t) - \tilde g_m(t) \right| \leq \left(\int_0^t \big|(D_m)_l(s) - D_m(s)\big|~ds\right) \mathrm{e}^{C(L,m) t^\ast}.
\end{equation*}

\end{proof}

\textbf{Acknowledgements}\hspace{1em} We acknowledge the financial support by the DAAD with funds of the German Federal Ministry of Education and Research (BMBF) through grant ID-57134585.

\bibliographystyle{amsalpha}
\bibliography{references}

\end{document}